\documentclass[12pt]{amsart}

\usepackage{geometry,graphicx,epsfig}
\usepackage{amsmath,amssymb,amsthm,color}
\usepackage{subfigure}
\usepackage{enumerate}
\usepackage{dcolumn}
\usepackage{mathtools} 

\usepackage[graph]{xy}

\usepackage[colorlinks=true,linkcolor=blue,citecolor=blue]{hyperref}

\newcommand{\tmmathbf}[1]{\ensuremath{\boldsymbol{#1}}}
\newcommand{\tmop}[1]{\ensuremath{\operatorname{#1}}}

\newcommand{\qede}{\hspace*{\fill}$\Diamond$\medskip}
\theoremstyle{plain}
  \newtheorem{theorem}{Theorem}
  \newtheorem{lemma}{Lemma}
  
  \newtheorem{corollary}{Corollary}
  
\theoremstyle{definition}
  \newtheorem{example}{Example}
  \newtheorem{remark}{Remark}
\renewcommand{\Re}[1]{\tmop{Re} #1}
\renewcommand{\ge}{\geqslant}
\renewcommand{\le}{\leqslant}
\newcommand{\pFq}[5]{\ensuremath{{}_{#1}F_{#2} \left( \genfrac{}{}{0pt}{}{#3}{#4} \bigg| {#5} \right)}}

\newcommand{\ipFq}[5]{\ensuremath{{}_{#1}F_{#2} \left( {#3} ; {#4} ; {#5} \right)}}

\newcommand{\MellinS}[2]{\ensuremath{\mathcal{M} \left[ {#1} ; {#2} \right]}}
\newcommand{\ift}{\int_0^\infty}
\newcommand{\md}{\mathrm{d}}
\newcommand{\id}{\, \md}
\newcommand{\Cl}{\tmop{Cl}}
\newcommand{\Res}{\tmop{Res}}
\newcommand{\QQ}{\mathbb{Q}}

\begin{document}

\title{Densities of short uniform random walks}


\author{Jonathan M. Borwein}
\address{CARMA, University of Newcastle, Australia}
\curraddr{}
\email{jonathan.borwein@newcastle.edu.au}
\thanks{}

\author{Armin Straub}
\address{Tulane University, USA}
\curraddr{}
\email{astraub@tulane.edu}
\thanks{}

\author{James Wan}
\address{CARMA, University of Newcastle, Australia}
\curraddr{}
\email{james.wan@newcastle.edu.au}
\thanks{}

\author{Wadim Zudilin}
\address{CARMA, University of Newcastle, Australia}
\curraddr{}
\email{wadim.zudilin@newcastle.edu.au}
\thanks{}

\dedicatory{{\rm with an appendix by: \uppercase{Don Zagier}}}
\address{Max-Planck-Institut f\"ur Mathematik, Bonn, Germany}
\curraddr{}
\email{don.zagier@mpim-bonn.mpg.de}




\maketitle

\begin{abstract}
  We study the densities of uniform random walks in the plane. A special focus
  is on the case of short walks with three or four steps and less completely
  those with five steps. As one of the main results, we obtain a hypergeometric
  representation of the density for four steps, which complements the classical
  elliptic representation in the case of three steps.  It appears unrealistic
  to expect similar results for more than five steps. New results are also
  presented concerning the moments of uniform random walks and, in particular,
  their derivatives. Relations with Mahler measures are discussed.
\end{abstract}

\section{Introduction}\label{sec:intro}

An $n$-step uniform random walk is a walk in the plane that starts at the origin and
consists of $n$ steps of length $1$ each taken into a uniformly random
direction. The study of such walks largely originated with Pearson more than a
century ago \cite{Pea05,Pea05b,Pea06} who posed the problem of determining the
distribution of the distance from the origin after a certain number of steps.
In this paper, we study the (radial) densities $p_n$ of the distance travelled in $n$
steps. This continues research commenced in \cite{bnsw-rw, bsw-rw2} where the focus
was on the moments of these distributions:
\begin{equation*}
  W_n(s):=\int_0^n p_n(t) t^{s} \id t.
\end{equation*}

The densities for walks of up to $8$ steps are depicted in Figure \ref{fig:pn}.
As established by Lord Rayleigh \cite{Ray05}, $p_n$ quickly approaches the
probability density $\frac{2 x}{n} e^{-x^2/n}$ for large $n$.  This limiting
density is superimposed in Figure \ref{fig:pn} for $n\ge5$.

\begin{figure}[htbp]
  \begin{center}
    \subfigure[$p_3$\label{fig:p3}]{\includegraphics[width=0.3\textwidth]{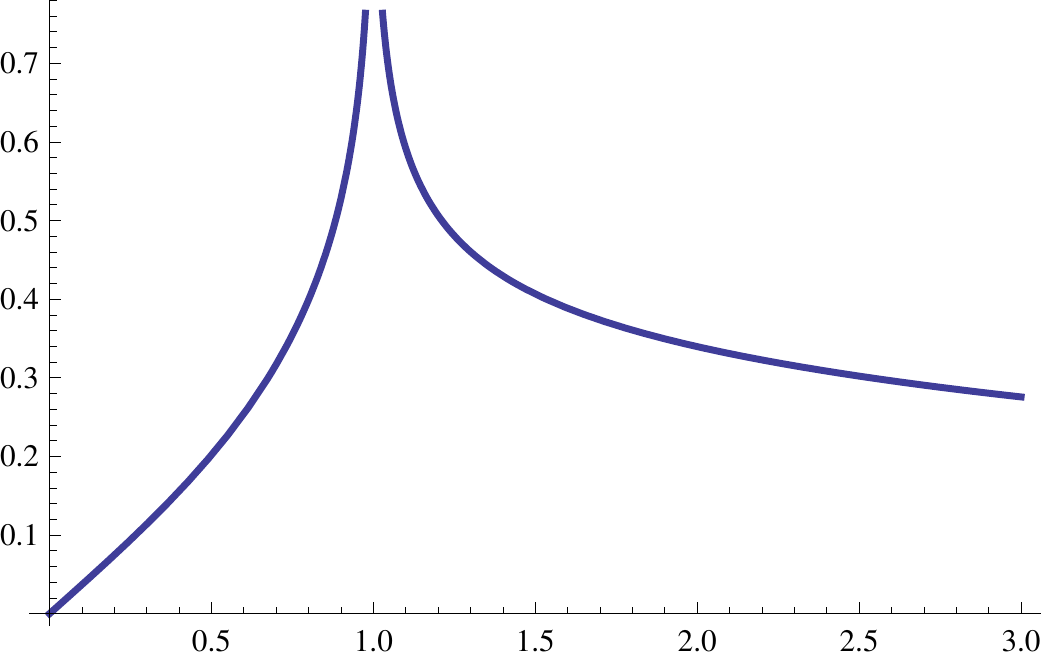}}
    \hfil
    \subfigure[$p_4$\label{fig:p4}]{\includegraphics[width=0.3\textwidth]{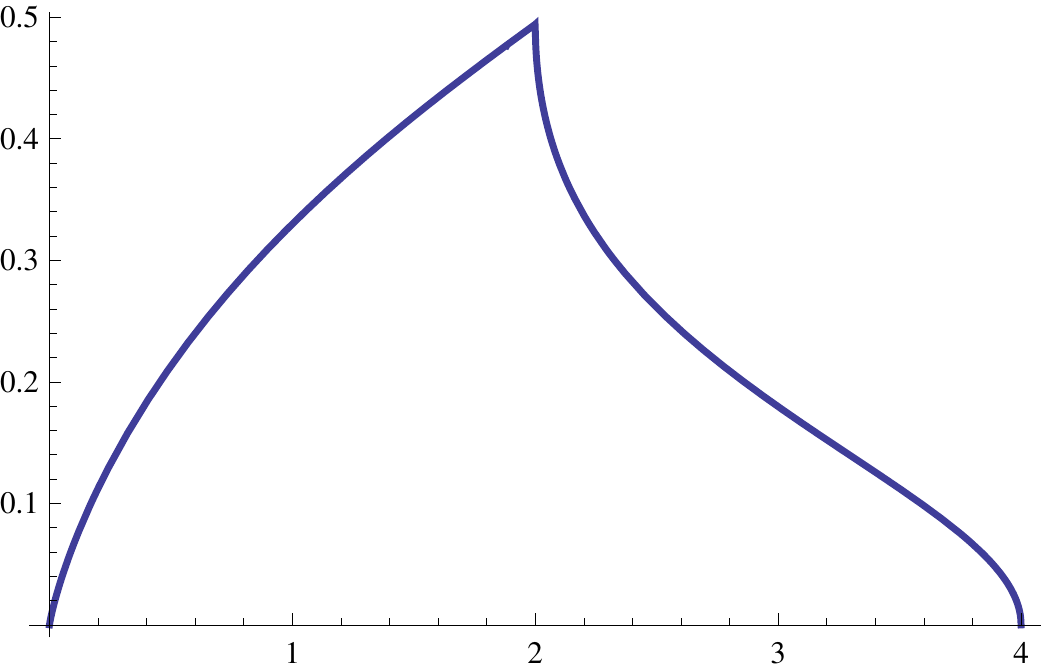}}
    \hfil
    \subfigure[$p_5$\label{fig:p5}]{\includegraphics[width=0.3\textwidth]{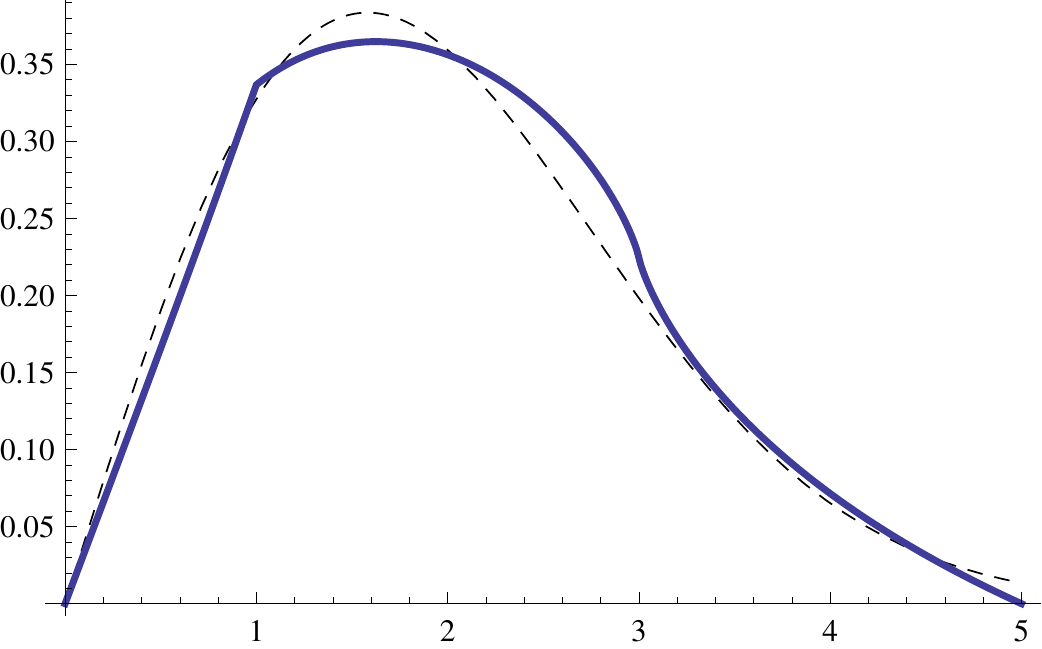}}\\
    \subfigure[$p_6$\label{fig:p6}]{\includegraphics[width=0.3\textwidth]{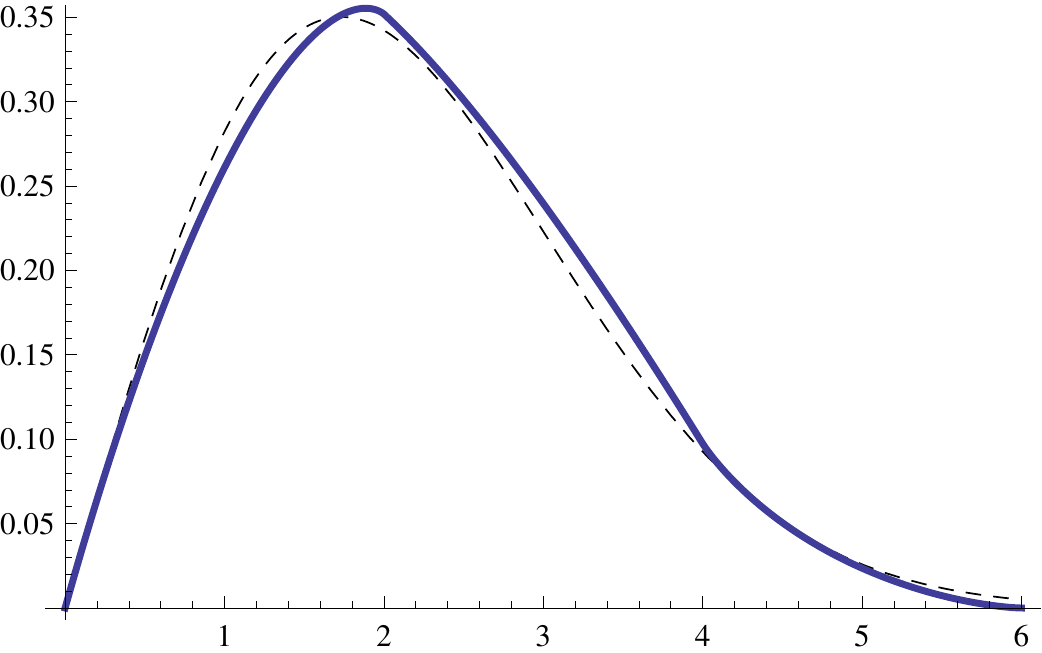}}
    \hfil
    \subfigure[$p_7$\label{fig:p7}]{\includegraphics[width=0.3\textwidth]{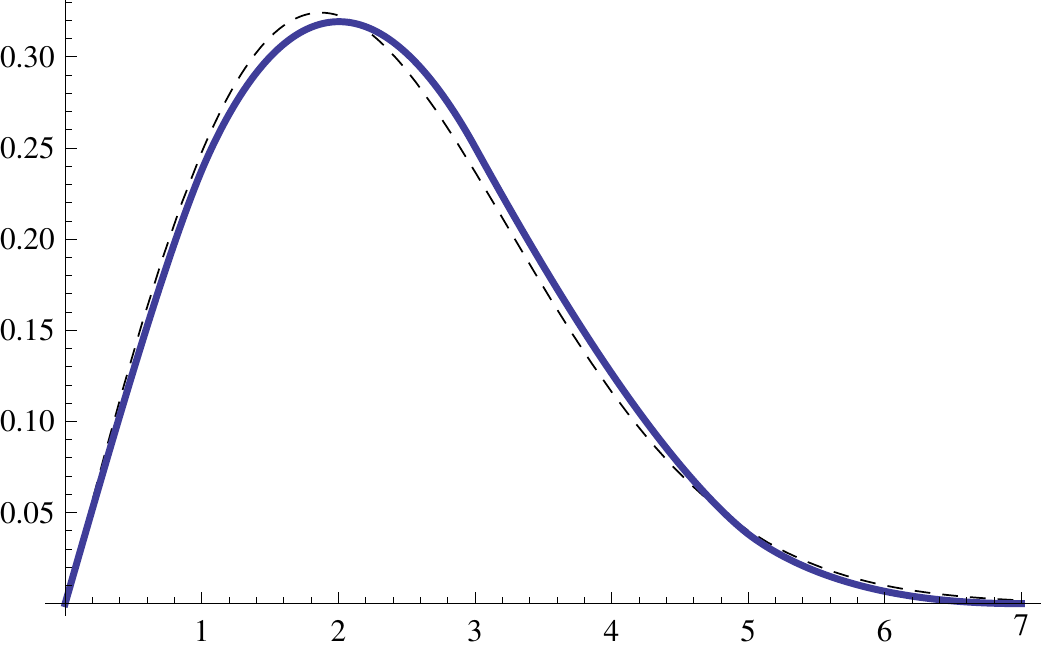}}
    \hfil
    \subfigure[$p_8$\label{fig:p8}]{\includegraphics[width=0.3\textwidth]{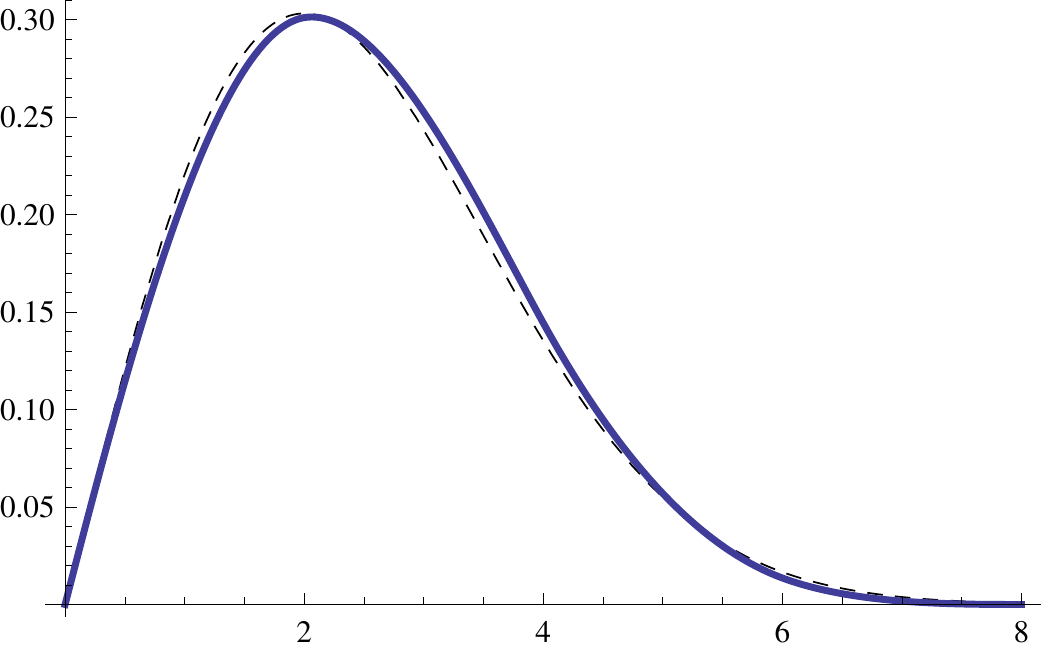}}
    \caption{Densities $p_n$ with the limiting behaviour superimposed for $n\ge5$.}
    \label{fig:pn}
  \end{center}
\end{figure}

Closed forms were only known in the cases $n=2$ and $n=3$. The evaluation, for $0\le x \le 2$,
\begin{equation}\label{eq:p2}
  p_2(x) =  \frac{2}{\pi \sqrt{4-x^2}}
\end{equation}
is elementary.  On the other hand, the density $p_3(x)$ for $0\le x \le 3$ can
be expressed in terms of elliptic integrals by
\begin{equation}\label{eq:p3ell}
  p_3(x) = \Re \left(\frac{\sqrt{x}}{\pi^2}\, K\left(\sqrt{\frac{(x+1)^3(3-x)}{16x}}\right)\right),
\end{equation}
see, e.g., \cite{Pea06}. One of the main results of this paper is a closed form
evaluation of $p_4$ as a hypergeometric function given in Theorem
\ref{thm:p4hyp04}.  In \eqref{eq:p3hyp} we also provide a single hypergeometric
closed form for $p_3$ which, in contrast to \eqref{eq:p3ell}, is real and valid
on all of $[0,3]$. For convenience, we list these two closed forms here:
\begin{align}
  p_3(x) &= \frac{2\sqrt{3}}{\pi}\,\frac{x}{\left(3+x^2\right)}
  \pFq21{\frac13,\frac23}{1}{\frac{x^2 \left( 9-x^2 \right)^2}{\left( 3+x^2 \right)^3}}, \\
  p_4(x) &= \frac2{\pi^2}\,\frac{\sqrt {16-x^2}}x\,
    \Re \pFq32{\frac12, \frac12, \frac12}{\frac56,\frac76}
      {\frac {\left( 16-x^2 \right)^3}{108 x^4}}.
\end{align}
We note that while \emph{Maple}  handled these well to high
precision, \emph{Mathematica} struggled, especially with the analytic
continuation of the $_3F_2$ when the argument is greater than 1.

A striking feature  of the 3- and 4-step random walk densities is
their modularity. It is this circumstance which not only allows us
to express them via hypergeometric series, but also makes them a
remarkable object of  mathematical study.

 This paper is structured as follows: In
Section \ref{sec:pn} we give general results for the densities $p_n$
and prove for instance that they satisfy certain linear differential
equations. In Sections \ref{sec:p3}, \ref{sec:p4}, and \ref{sec:p5}
we provide special results for $p_3$, $p_4$, and $p_5$ respectively.
Particular interest is taken in the behaviour near the points where
the densities fail to be smooth. In Section \ref{sec:Wndiff} we
study the derivatives of the moment function and make a connection
to multidimensional Mahler measures. Finally in Section \ref{sec:Wn}
we provide some related new evaluations of moments and so resolve a
central case of an earlier conjecture on convolutions of moments in
\cite{bsw-rw2}.

The amazing story of  the appearance of
Theorem~\ref{thm:deleadingcoeffx} is worth mentioning here. The
theorem was a conjecture in an earlier version of this manuscript,
and one of the present authors communicated it to D.~Zagier. That
author was surprised to learn that Zagier had already been asked for
a proof of exactly the same identities a little earlier, by
P.~Djakov and B.~Mityagin.

Those authors had in fact proved the theorem already in 2004 (see
\cite[Theorem 4.1]{dm1} and \cite[Theorem 8]{dm2}) during their
study of the asymptotics of the spectral gaps of a Schr\"odinger
operator with a two-term potential --- their proof was indirect, so
that we should  never have come across the identities without the
accident of asking the same person the same question! Djakov and
Mityagin asked Zagier about the possibility of a  direct proof of
their identities (the subject of Theorem~\ref{thm:deleadingcoeffx}),
and he gave a very neat  and purely combinatorial answer. It is this
proof which is herein presented in the Appendix.

We close this introduction with a historical remark illustrating the
fascination arising from these densities and their curious geometric features.
H.~Fettis devotes the entire paper \cite{fettis63} to proving that $p_5$ is not
linear on the initial interval $[0,1]$ as ruminated upon by Pearson
\cite{Pea06}. This will be explained in Section~\ref{sec:p5}.

\section{The densities $p_n$}\label{sec:pn}

It is a classical result of Kluyver \cite{Klu06} that $p_n$ has the following
Bessel integral representation:
\begin{equation}\label{eq:pnbessel}
  p_n(x) = \ift x t J_0(x t)J_0^n(t) \id t.
\end{equation}
Here $J_\nu$ is the \emph{Bessel function of the first kind} of order $\nu$.

\begin{remark}
  Equation \eqref{eq:pnbessel} naturally generalizes to the case of nonuniform
  step lengths.  In particular, for $n=2$ and step lengths $a$ and $b$ we
  record (see \cite[p. 411]{watson-bessel} or \cite[2.3.2]{hughes-rw}; the
  result is attributed to Sonine) that the
  corresponding density is
  \begin{align}\label{eq:p2ab}
    p_2(x;a,b) &= \ift x t J_0(x t)J_0(at)J_0(bt) \id t \nonumber\\
    &= \frac{2x}{\pi \sqrt{( (a+b)^2 - x^2 )( x^2 - (a-b)^2 )}}
  \end{align}
  for $|a-b| \le x \le a+b$ and $p_2(x;a,b)=0$ otherwise. Observe how \eqref{eq:p2ab}
  specializes to \eqref{eq:p2} in the case $a=b=1$.

  In the case $n=3$ the density $p_3(x;a,b,c)$ has been evaluated by Nicholson
  \cite[p.~414]{watson-bessel} in terms of elliptic integrals directly
  generalizing \eqref{eq:p3ell}. The corresponding extensions for four and more
  variables appear much less accessible.
  \qede
\end{remark}

It is visually clear from the graphs in Figure \ref{fig:pn} that $p_n$ is
getting smoother for increasing $n$. This can be made precise from
(\ref{eq:pnbessel}) using the asymptotic formula for $J_0$ for large arguments
and dominated convergence:

\begin{theorem} For  each integer $n \ge 0$,
  the density $p_{n+4}$ is $\lfloor n/2 \rfloor$ times continuously differentiable.
\end{theorem}

On the other hand, we note from Figure \ref{fig:pn} that the only
points preventing $p_n$ from being smooth appear to be integers.
This will be made precise in Theorem \ref{thm:pnde}.

To this end, we recall a few things about the $s$-th moments $W_n(s)$ of the
density $p_n$ which are given by
\begin{equation}\label{def:Wns}
  W_n (s) = \ift x^s p_n(x) \id x =  \int_{[0, 1]^n} \left| \sum_{k = 1}^n e^{2 \pi x_k i}
   \right|^s \id \tmmathbf{x}.
\end{equation}
Starting with the right-hand side, these moments had been
investigated in \cite{bnsw-rw, bsw-rw2}.  There it was shown that
$W_n(s)$ admits an \textit{analytic continuation} to all of the
complex plane with poles of at most order two at certain negative
integers. In particular, $W_3(s)$ has simple poles at $s=-2, -4, -6,
\ldots$ and $W_4(s)$ has double poles at these integers \cite[Thm.
6, Ex. 2 \& 3]{bnsw-rw}.

Moreover, from the combinatorial evaluation
\begin{equation}\label{eq:Wn-even}
  W_n (2 k) = \sum_{a_1 + \cdots + a_n = k} \binom{k}{a_1, \ldots, a_n}^2
\end{equation}
for integers $k\ge0$ it followed that $W_n(s)$ satisfies a
functional equation, as in \cite[Ex. 1]{bnsw-rw}, coming from the
inevitable recursion that exists for the right-hand side of
(\ref{eq:Wn-even}) . For instance,
\[ (s + 4)^2 W_3 (s + 4) - 2 (5 s^2 + 30 s + 46) W_3 (s + 2) + 9 (s + 2)^2 W_3 (s) = 0,\]
and equation (\ref{eq:W4fe}) below.

 The first part of equation
(\ref{def:Wns}) can be rephrased as saying that $W_n(s-1)$ is the
\emph{Mellin transform} of $p_n$ (\cite{ml-xmellin}). We denote this
by $W_n(s-1)=\MellinS{p_n}{s}$. Conversely, the density $p_n$ is the
\emph{inverse Mellin transform} of $W_n(s-1)$. We intend to exploit
this relation as detailed for $n=4$ in the following example.

\begin{example}[Mellin transforms]\label{eg:p4de}
  For $n=4$, the moments $W_4(s)$ satisfy the functional equation
  \begin{equation}\label{eq:W4fe}
    (s+4)^3 W_4(s+4) - 4(s+3)(5s^2+30s+48) W_4(s+2) + 64(s+2)^3 W_4(s) = 0.
  \end{equation}
  Recall the following rules for the Mellin transform: if $F (s) = \MellinS{f}{s}$
  then in the appropriate strips of convergence
  \begin{itemize}
    \item $\MellinS{x^{\mu} f (x)}{s} = F (s + \mu)$,
    \item $\MellinS{D_x f (x)}{s} = - (s - 1) F (s - 1)$.
  \end{itemize}
  Here, and below, $D_x$ denotes differentiation with respect to $x$, and, for
  the second rule to be true, we have to assume, for instance, that $f$ is
  continuously differentiable.

  Thus, purely formally, we can translate the functional equation
  (\ref{eq:W4fe}) of $W_4$ into the differential equation $A_4 \cdot p_4 (x) =
  0$ where $A_4$ is the operator
  \begin{align}\label{eq:p4de}
    A_4 &= x^4 (\theta+1)^3 - 4x^2 \theta (5\theta^2+3) + 64(\theta-1)^3 \\
    &= (x - 4) (x - 2) x^3 (x + 2) (x + 4) D_x^3 + 6 x^4
    \left( x^2 - 10 \right) D_x^2 \label{eq:p4deD}\\
    &\quad+ x \left( 7 x^4 - 32 x^2 + 64 \right) D_x + \left(
    x^2 - 8 \right)  \left( x^2 + 8 \right) . \nonumber
  \end{align}
  Here $\theta=x D_x$.
  However, it should be noted that $p_4$ is not continuously differentiable.
  Moreover, $p_4(x)$ is approximated by a constant multiple of $\sqrt{4-x}$
  as $x\to4^-$ (see Theorem \ref{thm:p4at4}) so that the second derivative of $p_4$ is
  not even locally integrable. In particular, it does not have a Mellin
  transform in the classical sense.
\qede
\end{example}

\begin{theorem}\label{thm:pnde}
  Let an integer $n\ge1$ be given.
  \begin{itemize}
    \item The density $p_n$ satisfies a differential equation of order $n - 1$.

    \item If $n$ is even (respectively odd) then $p_n$ is real analytic except
      at $0$ and the even (respectively odd) integers $m \leq n$.
  \end{itemize}
\end{theorem}

\begin{proof}
  As illustrated for $p_4$ in Example \ref{eg:p4de}, we formally use the Mellin
  transform method to translate the functional equation of $W_n$ into a
  differential equation $A_n \cdot y(x) = 0$. Since $p_n$ is locally integrable
  and compactly supported, it has a Mellin transform in the distributional
  sense as detailed for instance in \cite{ml-xmellin}. It follows rigorously that $p_n$
  solves $A_n \cdot y(x) = 0$ in a distributional sense.  In other words, $p_n$
  is a weak solution of this differential equation. The degree of this equation
  is $n - 1$ because the functional equation satisfied by $W_n$ has
  coefficients of degree $n - 1$ as shown in \cite[Thm. 1]{bnsw-rw}.

  The leading coefficient of the differential equation (in terms of $D_x$ as in
  \eqref{eq:p4deD}) turns out to be
  \begin{equation}\label{eq:deleadingcoeff}
    x^{n - 1} \prod_{2|(m-n)} (x^2 - m^2)
  \end{equation}
  where the product is over the even or odd integers $1 \leq m \leq n$
  depending on whether $n$ is even or odd. This is discussed below
  in Section \ref{sec:deleadingcoeff}.

  Thus the leading coefficient of the differential equation is nonzero on $[0,
  n]$ except for $0$ and the even or odd integers already mentioned. On each interval
  not containing these points it follows, as described for instance in
  \cite[Cor. 3.1.6]{hoermander-lpde1}, that $p_n$ is in fact a classical solution of
  the differential equation. Moreover the analyticity of the coefficients,
  which are polynomials in our case, implies that $p_n$ is piecewise real
  analytic as claimed.
\end{proof}

\begin{remark}\label{rk:pnat0}
  It is one of the basic properties of the Mellin transform, see for instance
  \cite[Appendix B.7]{anacomb}, that the asymptotic behaviour of a function at
  zero is determined by the poles of its Mellin transform which lie to the left
  of the fundamental strip. It is shown in \cite{bnsw-rw} that the poles of
  $W_n(s)$ occur at specific negative integers and are at most of second order.
  This translates into the fact that $p_n$ has an expansion at $0$ as a power
  series with additional logarithmic terms in the presence of double poles.
  This is made explicit in the case of $p_4$ in Example \ref{eg:W4p4mellinat0}.
\end{remark}

\subsection{An explicit recursion}\label{sec:deleadingcoeff}

We close this section by providing details for the claim made in
\eqref{eq:deleadingcoeff}.  Recall that the even moments $ f_n(k) := W_n (2 k)
$ satisfy a recurrence of order $\lambda:=\lceil n / 2 \rceil$ with polynomial
coefficients of degree $n - 1$ (see \cite{bnsw-rw}). An entirely explicit
formula for this recurrence is given in \cite{verrill}:

\begin{theorem}\label{thm:verrill}
  \begin{equation}\label{eq:recverrill}
    \sum_{j\ge0} \left[ k^{n+1} \sum_{\alpha_1,\ldots,\alpha_j}
      \prod_{i=1}^j (-\alpha_i)(n+1-\alpha_i) \left( \frac{k-i}{k-i+1} \right)^{\alpha_i-1}
      \right] f_n(k-j) = 0
  \end{equation}
  where the sum is over all sequences $\alpha_1,\ldots,\alpha_j$ such that $0\le\alpha_i\le n$
  and $\alpha_{i+1}\le\alpha_i-2$.
\end{theorem}

Observe that \eqref{eq:deleadingcoeff} is easily checked for each fixed $n$ by
applying Theorem \ref{thm:verrill}. We explicitly checked the cases $n\le1000$
(using a recursive formulation of Theorem \ref{thm:verrill} from
\cite{verrill}) while only using this statement for $n\le5$ in this paper. The
fact that \eqref{eq:deleadingcoeff} is true in general is recorded and made
more explicit in Theorem \ref{thm:deleadingcoeffx} below.

For fixed $n$, write the recurrence for $f_n(k)$ in the form
$\sum_{j=0}^{n-1} k^j q_j(K)$ where $q_j$ are polynomials and $K$ is the
shift $k \to k+1$. Then $q_{n-1}$ is the characteristic polynomial of this
recurrence, and, by the rules outlined in Example \ref{eg:p4de}, we find that
the differential equation satisfied by $p_n(x)$ is of the form $ q_{n-1}(x^2)
\theta^{n-1} + \cdots, $ where $\theta = x D_x$ and the dots indicate terms of
lower order in $\theta$.

We claim that the characteristic polynomial of the recurrence
\eqref{eq:recverrill} satisfied by $f_n(k)$ is $ \prod_{2|(m-n)} (x - m^2) $
where the product is over the integers $1 \leq m \leq n$ such that $m\equiv n$
modulo $2$.  This implies \eqref{eq:deleadingcoeff}.
By Theorem \ref{thm:verrill} the characteristic polynomial is
\begin{equation}\label{eq:lcverrill}
  \sum_{j=0}^\lambda \left[ \sum_{\alpha_1,\ldots,\alpha_j}
    \prod_{i=1}^j (-\alpha_i)(n+1-\alpha_i) \right] x^{\lambda-j}
\end{equation}
where $\lambda=\lceil n / 2 \rceil$ and the sum is again over all
sequences $\alpha_1,\ldots,\alpha_j$ such that $0\le\alpha_i\le n$
and $\alpha_{i+1}\le\alpha_i-2$. The claimed evaluation is thus
equivalent to the following identity, first proven by P.~Djakov and
B.~Mityagin \cite{dm1,dm2}. Zagier's more direct and purely
combinatorial proof  is given in the Appendix.

\begin{theorem}\label{thm:deleadingcoeffx}
  For all integers $n,j\ge1$,
  \begin{equation}\label{eq:exconj}
   \sum_{ 0 \leqslant m_1, \ldots, m_j < n / 2 \atop  m_i < m_{i + 1} }
      \prod_{i = 1}^j (n - 2 m_i)^2
     = \sum_{ 1 \leqslant \alpha_1, \ldots, \alpha_j \leqslant n \atop
      \alpha_i \leqslant \alpha_{i + 1} - 2 }
      \prod_{i = 1}^j \alpha_i (n + 1 - \alpha_i).
  \end{equation}
\end{theorem}


\section{The density $p_3$}\label{sec:p3}

The elliptic integral evaluation \eqref{eq:p3ell} of $p_3$ is very suitable to
extract information about the features of $p_3$ exposed in Figure \ref{fig:p3}.
It follows, for instance, that $p_3$ has a singularity at $1$. Moreover, using
the known asymptotics for $K(x)$, we may deduce that the singularity is of the
form
\begin{equation}\label{eq:p3at1}
  p_3 (x) = \frac{3}{2\pi^2}\log\left( \frac{4}{|x-1|} \right) + O (1)
\end{equation}
as $x\to1$.

We also recall from \cite[Ex. 5]{bsw-rw2} that $p_3$ has the
expansion, valid for $0\le x \le1$,
\begin{equation}\label{eq:p3at0}
  p_3(x) = \frac{2 x}{\pi\sqrt{3}}\sum_{k=0}^\infty W_3(2k) \left(\frac x3\right)^{2k}
\end{equation}
where
\begin{equation}
  W_3(2k)= \sum_{j=0}^k  \binom{k}{j}^2 \binom{2j}{j}
\end{equation}
is the sum of squares of trinomials. Moreover, we have from \cite[Eqn.
29]{bsw-rw2} the functional relation
\begin{equation}\label{eq:p3mod}
  p_3(x)=\frac{4x}{(3-x)(x+1)}\,p_3\left(\frac{3-x}{1+x}\right)
\end{equation}
so that \eqref{eq:p3at0} determines $p_3$ completely and also makes apparent the
behaviour at $3$.

We close this section with two more alternative expressions for $p_3$.

\begin{example}[Hypergeometric form for $p_3$]
  Using the techniques in \cite{cz-apery} we can deduce from \eqref{eq:p3at0} that
  \begin{equation}\label{eq:p3hyp}
    p_3(x) =\frac{ 2\,\sqrt{3}x}{\pi  \left( 3+{x}^{2} \right)}\,
  {_2F_1\left(\frac13,\frac23;\,1;\,{\frac {{x}^{2} \left( 9-x^2 \right)^2}{ \left( 3+x^2 \right)^3}}\right)}
  \end{equation}
  which is found in a similar but simpler way than the hypergeometric form
  of $p_4$ given in Theorem \ref{thm:p4hyp04}.
  Once obtained, this identity is easily proven using the differential equation
  from Theorem \ref{thm:pnde} satisfied by $p_3$.  From \eqref{eq:p3hyp} we
  see, for example, that $p_3(\sqrt{3})^2 = \frac{3}{2\pi^2}W_3(-1)$.
\qede
\end{example}

\begin{example}[Iterative form for $p_3$]\label{ex:agmp3}
  The expression \eqref{eq:p3hyp} can be interpreted in terms of the cubic AGM,
  $\tmop{AG}_3$, see \cite{bb-cubicagm}, as follows. Recall that
  $\tmop{AG}_3(a,b)$ is the limit of iterating
  \begin{equation*}
    a_{n+1}=\frac{a_n+2b_n}{3}, \quad b_{n+1}=\sqrt[3]{b_n \left( \frac{a_n^2+a_n b_n+b_n^2}{3} \right)},
  \end{equation*}
  beginning with $a_0=a$ and $b_0=b$. The iterations converge cubically, thus allowing for very
  efficient high-precision evaluation. On the other hand,
  \begin{equation*}
    \frac{1}{\tmop{AG}_3(1,s)} = \ipFq21{\frac13,\frac23}{1}{1-s^3}
  \end{equation*}
  so that in consequence of \eqref{eq:p3hyp}, for $0\le x \le3$,
  \begin{equation}\label{eq:p3ag3}
    p_3(x) = \frac{2\sqrt 3}{\pi} \,\frac{x}{\tmop{AG}_3(3+x^2,3\left|1-x^2\right|^{2/3})}.
  \end{equation}
  Note that $p_3(3)=\frac{\sqrt{3}}{2\pi}$ is a direct consequence of the final formula.

  Finally we remark that the cubic AGM also makes an appearance in the case
  $n=4$.  We just mention that the modular properties of $p_4$ recorded in
  Remark~\ref{rk:p4modularity} can be stated in terms of the theta functions
  \begin{equation}
    b(\tau) = \frac{\eta(\tau)^3}{\eta(3\tau)}, \quad
    c(\tau) = 3\frac{\eta(3\tau)^3}{\eta(\tau)}
  \end{equation}
  where $\eta$ is the Dedekind eta function defined in \eqref{eq:eta}.
  For more information and proper definitions of the functions $b$, $c$ as well
  as $a$, which is related by $a^3 = b^3 + c^3$, we refer to \cite{garvan}.
  Ultimately we are hopeful that, in search for an analogue of \eqref{eq:p3mod}
  for $p_4$, this may lead to an algebraic relation between algebraically
  related arguments of $p_4$.
\qede
\end{example}

\section{The density $p_4$}\label{sec:p4}

The densities $p_n$ are recursively related. As in \cite{hughes-rw}, setting
$\phi_{n}(x) =  p_{n}(x)/(2\pi x)$, we have that for integers $n\ge2$
\begin{equation}\label{eq:pnrec}
  \phi_n(x) = \frac{1}{2\pi}\, \int_0^{2\pi}
    \phi_{n-1}\left(\sqrt{x^2-2x \cos \alpha + 1}\right) \id \alpha.
\end{equation}

We use this recursive relation to get some quantitative information about the
behaviour of $p_4$ at $x=4$.

\begin{theorem}\label{thm:p4at4}
  As $x \rightarrow 4^-$,
  \[ p_4(x) = \frac{\sqrt{2}}{\pi^2}\sqrt{4 - x} - \frac{3\sqrt{2}}{16\pi^2} (4-x)^{3/2}
     + \frac{23\sqrt{2}}{512\pi^2} (4-x)^{5/2} + O\left( (4-x)^{7/2} \right). \]
\end{theorem}

\begin{proof}
  Set $y=\sqrt{x^2-2x \cos \alpha + 1}$. For $2 < x < 4$,
  \[ \phi_4 (x) = \frac{1}{\pi} \int_0^{\pi} \phi_3 (y)\, \md \alpha  =
     \frac{1}{\pi} \int_0^{\arccos ( \frac{x^2 - 8}{2 x})} \phi_3 (y)\, \md \alpha
      \]
  since $\phi_3$ is only supported on $[0, 3]$. Note that $\phi_3 (y)$ is
  continuous and bounded in the domain of integration. By the Leibniz integral
  rule, we can thus differentiate under the integral sign to obtain
  \begin{equation}\label{eq:phi4D1}
    \phi'_4 (x) = - \frac{1}{\pi}  \frac{(x^2 + 8)\, \phi_3 (3)}{x \sqrt{(16 -
    x^2) (x^2 - 4)}} + \frac{1}{\pi} \int_0^{\arccos ( \frac{x^2 - 8}{2 x})}
    (x - \cos (\alpha)) \,\frac{\phi_3' (y)}{y}\, \md \alpha.
  \end{equation}
  This shows that $\phi_4'$, and hence $p_4'$, have a singularity at $x =
  4$. More specifically,
  \[ \phi_4' (x) = - \frac{1}{8 \sqrt{2} \pi^3 \sqrt{4 - x}} + O (1)
     \hspace{1em} \text{as $x \rightarrow 4^-$} . \]
  Here, we used that $\phi_3 (3) = \frac{\sqrt{3}}{12 \pi^2}$.
  It follows that
  \[ p_4' (x) = - \frac{1}{\sqrt{2} \pi^2 \sqrt{4 - x}} + O (1) \]
  which, upon integration, is the claim to first order. Differentiating
  \eqref{eq:phi4D1} twice more proves the claim.
\end{proof}

\begin{remark}
  The situation for the singularity at $x = 2^+$ is more complicated since in
  (\ref{eq:phi4D1}) both the integral (via the logarithmic singularity of
  $\phi_3$ at $1$, see \eqref{eq:p3at1}) and the boundary term contribute.
  Numerically, we find, as $x\to2^+$,
  \[ p_4' (x) = - \frac{2}{\pi^2 \sqrt{x - 2}} + O (1). \]
  On the other hand, the derivative of $p_4$ at 2 from the left is given by
  \[ p_4'(2^-) = \frac{\sqrt3}{\pi} \, \pFq32{-\frac12,\frac13,\frac23}{1,1}{1}
    - \frac23\, p_4(2). \]
  These observations can be proven in hindsight from Theorem \ref{thm:p4hyp}.
\qede
\end{remark}

We now turn to the behaviour of $p_4$ at zero which we derive from
the pole structure of $W_4$ as described in Remark \ref{rk:pnat0}.

\begin{figure}[htbp]\label{fig:W45}
  \begin{center}
    \subfigure[$W_4$\label{fig:W4}]{\includegraphics[width=0.4\textwidth]{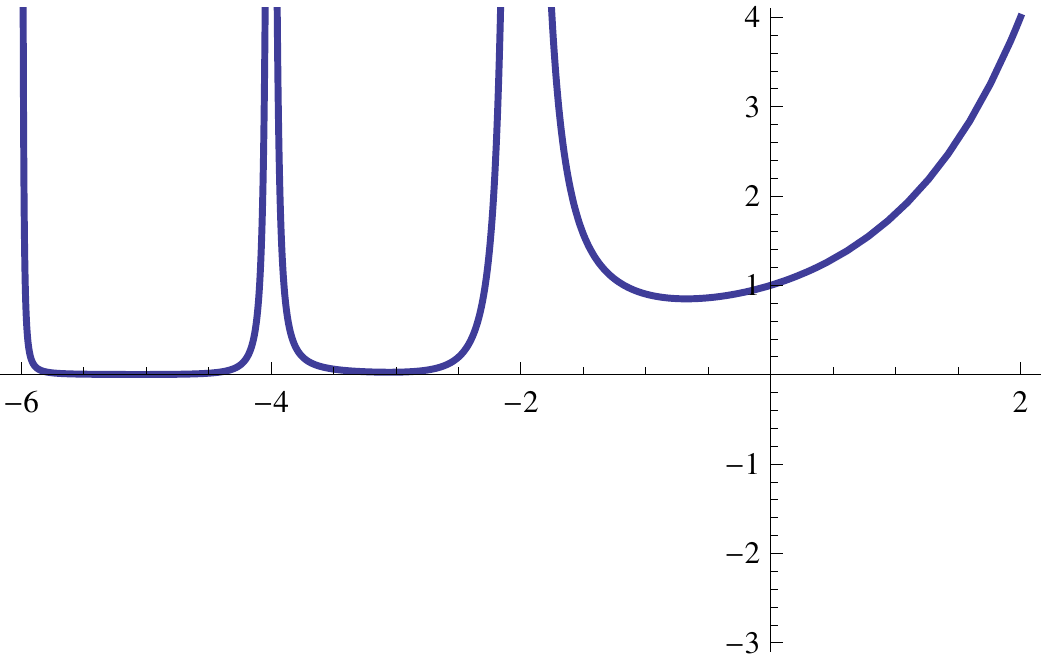}}
    \hfil
    \subfigure[$W_5$\label{fig:W5}]{\includegraphics[width=0.4\textwidth]{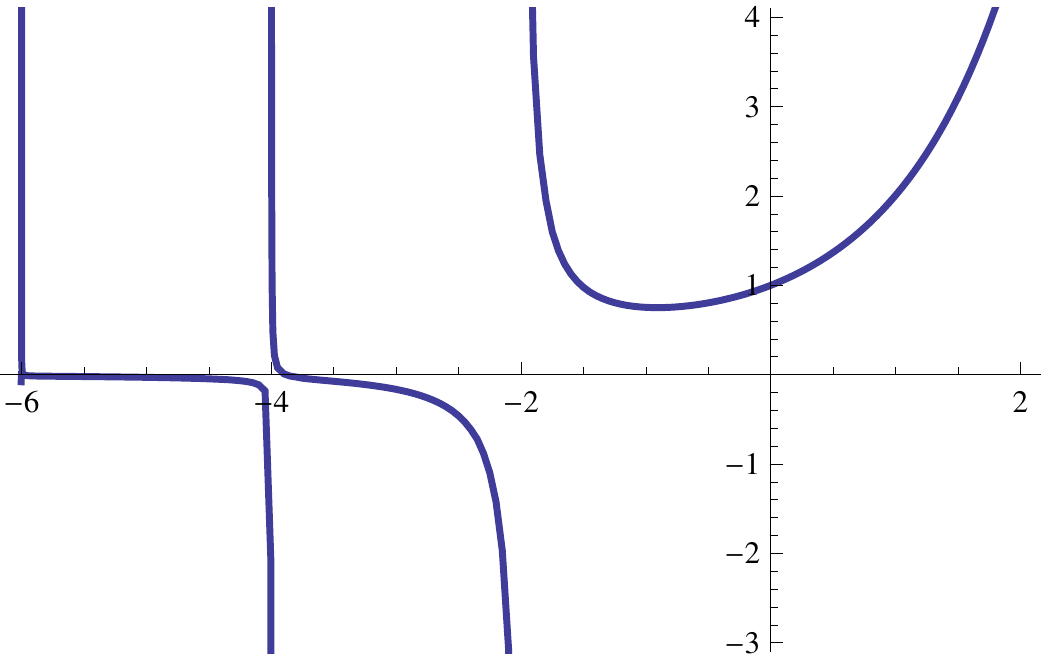}}
    \caption{$W_4$ and $W_5$ analytically continued to the real line.}
  \end{center}
\end{figure}

\begin{example}\label{eg:W4p4mellinat0}
  From \cite{bsw-rw2}, we know that $W_4$ has a pole of order $2$ at $- 2$ as
  illustrated in Figure \ref{fig:W4}. More specifically, results in Section \ref{sec:Wndiff} give
  \[ W_4 (s) = \frac{3}{2 \pi^2}  \frac{1}{(s + 2)^2} + \frac{9}{2 \pi^2} \log
     (2) \frac{1}{s + 2} + O (1) \]
  as $s \rightarrow - 2$. It therefore follows that
  \[ p_4 (x) = - \frac{3}{2 \pi^2} x \log (x) + \frac{9}{2 \pi^2} \log (2) x + O(x^3) \]
  as $x \rightarrow 0$.
\qede
\end{example}

More generally, $W_4$ has poles of order $2$ at $- 2 k$ for $k$ a positive
integer. Define $s_{4,k}$ and $r_{4,k}$ by
\begin{equation}\label{eq:rs4def}
  W_4 (s) = \frac{s_{4, k-1}}{(s + 2 k)^2} + \frac{r_{4, k-1}}{s + 2 k} + O (1)
\end{equation}
as $s \rightarrow - 2 k$. We thus obtain that, as $x \rightarrow 0^+$,
\begin{equation*}
  p_4 (x) = \sum_{k = 0}^{K-1} x^{2 k + 1} \left( r_{4, k} - s_{4, k} \log (x)
   \right) + O (x^{2 K + 1}).
\end{equation*}
In fact, knowing that $p_4$ solves the linear Fuchsian differential equation
\eqref{eq:p4de} with a regular singularity at $0$ we may conclude:

\begin{theorem}\label{thm:p4at0}
  For small values $x>0$,
  \begin{equation}\label{eq:p4at0}
    p_4 (x) = \sum_{k = 0}^\infty \left( r_{4, k} - s_{4, k} \log (x)
     \right)\,x^{2 k + 1}.
  \end{equation}
\end{theorem}

Note that
\[ s_{4, k} = \frac{3}{2 \pi^2}  \frac{W_4 (2 k)}{8^{2 k}} \]
as the two sequences satisfy the same recurrence and initial conditions.  The
numbers $W_4(2k)$ are also known as the Domb numbers (\cite{bbbg-bessel}), and
their generating function in hypergeometric form is given in \cite{rogers-5f4}
and has been further studied in \cite{cz-apery}. We thus have
\begin{equation}\label{eq:p4de-ana}
  \sum_{k=0}^\infty s_{4,k} \,x^{2k+1} = \frac{6x}{\pi^2 \left( 4-x^2 \right)}\,
    \pFq32{\frac13,\frac12,\frac23}{1,1}{\frac {108x^2}{ \left( x^2-4 \right)^3}}
\end{equation}
which is readily verified to be an analytic solution to the differential
equation satisfied by $p_4$.

\begin{remark}
  For future use, we note that \eqref{eq:p4de-ana} can also be written as
  \begin{equation}\label{eq:p4de-ana2}
    \sum_{k=0}^\infty s_{4,k} \,x^{2k+1} = \frac{24x}{\pi^2 \left( 16-x^2 \right)}\,
      \pFq32{\frac13,\frac12,\frac23}{1,1}{\frac {108x^4}{ \left( 16-x^2 \right)^3}}
  \end{equation}
  which follows from the transformation
  \begin{equation}\label{eq:hyptr32}
    (1-4x)  \pFq32{\frac13,\frac12,\frac23}{1,1}{-\frac{108x}{(1-16x)^3}} =
    (1-16x) \pFq32{\frac13,\frac12,\frac23}{1,1}{\frac{108x^2}{(1-4x)^3}}
  \end{equation}
  given in \cite[(3.1)]{cz-apery}.
  \qede
\end{remark}

On the other hand, as a consequence of \eqref{eq:rs4def} and the functional
equation \eqref{eq:W4fe} satisfied by $W_4$, the residues $r_{4,k}$ can be
obtained from the recurrence relation
\begin{align}
  128 k^3 r_{4,k} &= 4(2k-1)(5k^2-5k+2) r_{4,k-1}
  -2(k-1)^3 r_{4,k-2} \nonumber\\
  &\quad+3\left( 64k^2 s_{4,k} - (20k^2-20k+6) s_{4,k-1} + (k-1)^2 s_{4,k-2} \right)
\end{align}
with $r_{4,-1}=0$ and $r_{4,0}=\frac{9}{2\pi^2}\log(2)$.

\begin{remark}
  In fact, before realizing the connection between the Mellin transform and the
  behaviour of $p_4$ at $0$, we empirically found that $p_4$ on $(0,2)$ should
  be of the form $r(x)-s(x)\log(x)$ where $a$ and $r$ are odd and analytic. We
  then numerically determined the coefficients and observed the relation with
  the residues of $W_4$ as given in Theorem \ref{thm:p4at0}.
  \qede
\end{remark}

The differential equation for $p_4$ has a regular singularity at $0$. A basis
of solutions at $0$ can therefore be obtained via the Frobenius method, see for
instance \cite{ince-ode}. Since the indicial equation has $1$ as a triple root, the
solution \eqref{eq:p4de-ana} is the unique analytic solution at $0$ while the
other solutions have a logarithmic or double logarithmic singularity. The
solution with a logarithmic singularity at $0$ is explicitly given in
\eqref{eq:p4de-logsol}, and, from \eqref{eq:p4at0}, it is clear that $p_4$ on
$(0,2)$ is a linear combination of the analytic and the logarithmic solution.

Moreover, the differential equation for $p_4$ is a symmetric square. In other
words, it can be reduced to a second order differential equation, which after a
quadratic substitution, has 4 regular singularities and is thus of Heun type. In
fact, a hypergeometric representation of $p_4$ with rational argument is possible.

\begin{theorem}\label{thm:p4hyp}
  For $2<x<4$,
  \begin{equation}\label{eq:p4hyp}
    p_4(x) = \frac2{\pi^2}\,\frac{\sqrt {16-x^2}}x\,
      \pFq32{\frac12, \frac12, \frac12}{\frac56,\frac76}
        {\frac {\left( 16-x^2 \right)^3}{108 x^4}}.
  \end{equation}
\end{theorem}

\begin{proof}
  Denote the right-hand side of \eqref{eq:p4hyp} by $q_4(x)$ and observe that
  the hypergeometric series converges for $2<x<4$. It is routine to verify that
  $q_4$ is a solution of the differential equation $A_4 \cdot y(x) = 0$ given
  in \eqref{eq:p4de} which is also satisfied by $p_4$ as proven in Theorem
  \ref{thm:pnde}. Together with the boundary conditions supplied by Theorem
  \ref{thm:p4at4} it follows that $p_4=q_4$.
\end{proof}

We note that Theorem \ref{thm:p4hyp} gives $2\sqrt{16-x^2}/(\pi^2 x)$ as an approximation to $p_4(x)$ near $x=4$,  which is much more accurate than the elementary estimates established in Theorem \ref{thm:p4at4}.

\begin{corollary}
  In particular,
  \begin{equation}
    p_4(2)=\frac {2^{7/3}\pi}{3\sqrt{3}}\,\Gamma\left (\frac 23 \right)^{-6} = \frac{\sqrt 3}{\pi}\,W_3(-1).
  \end{equation}
\end{corollary}

Quite marvelously, as first discovered numerically:

\begin{theorem}\label{thm:p4hyp04}
  For $0<x<4$,
  \begin{equation}\label{eq:p4hyp04}
    p_4(x) = \frac2{\pi^2}\,\frac{\sqrt {16-x^2}}x\,
      \Re \pFq32{\frac12, \frac12, \frac12}{\frac56,\frac76}
        {\frac {\left( 16-x^2 \right)^3}{108 x^4}}.
  \end{equation}
\end{theorem}

\begin{proof}
  To obtain the analytic continuation of the ${}_3F_2$ for $0<x<2$ we employ
  the formula \cite[5.3]{luke-sf}, valid for all $z$,
  \begin{align*}
    \pFq{q+1}{q}{a_1,\ldots,a_{q+1}}{b_1,\ldots,b_q}{z} &= \frac{\prod_j\Gamma(b_j)}{\prod_j\Gamma(a_j)}
    \sum_{k=1}^{q+1} \frac{\Gamma(a_k)\prod_{j\ne k}\Gamma(a_j-a_k)}{\prod_j\Gamma(b_j-a_k)} (-z)^{-a_k} \\
      &\quad \times\pFq{q+1}{q}{a_k,\{a_k-b_j+1\}_j}{\{a_k-a_j+1\}_{j\ne k}}{\frac1z},
  \end{align*}
  which requires the $a_j$ to not differ by integers. Therefore we apply it to
  \begin{equation*}
    \pFq32{\frac12+\varepsilon,\frac12,\frac12-\varepsilon}{\frac56,\frac76}{z}.
  \end{equation*}
  and take the limit as $\varepsilon\to0$. This ultimately produces, for $z>1$,
  \begin{align}\label{eq:p4de-logsol}
    \Re \pFq32{\frac12, \frac12, \frac12}{\frac56, \frac76}{z}
      &= \frac{\log(108z)}{2\sqrt{3z}}\, \pFq32{\frac13, \frac12, \frac23}{1, 1}{\frac{1}{z}}\\
      &\quad+ \frac{1}{2\sqrt{3z}}\,\sum_{n=0}^\infty \frac{(\tfrac13)_n (\tfrac12)_n
      (\tfrac23)_n}{ n!^3} \left( \frac1z \right)^n  (5H_n-2H_{2n}-3H_{3n}).\nonumber
  \end{align}
  Here $H_n = \sum_{k=1}^n 1/k$ is the $n$-th harmonic number.  Now, insert the
  appropriate argument for $z$ and the factors so the left-hand side
  corresponds to the claimed closed form.  Observing that
  \begin{equation*}
    \left(\tfrac13\right)_n \left(\tfrac12\right)_n \left(\tfrac23\right)_n = \frac{(2n)!(3n)!}{108^n (n!)^2},
  \end{equation*}
  we thus find that the right-hand side of \eqref{eq:p4hyp04} is given by
  $-\log(x) S_4(x)$ plus
  \begin{align*}
    \frac{6}{\pi^2} \sum_{n=0}^\infty \frac{(2n)!(3n)!}{(n!)^5}
      \frac{x^{4n+1}}{(16-x^2)^{3n}} \left( 5H_n-2H_{2n}-3H_{3n} +3\log(16-x^2) \right) \nonumber
  \end{align*}
  where $S_4$ is the solution (analytic at $0$) to the differential equation
  for $p_4$ given in \eqref{eq:p4de-ana2}.  This combination can now be verified
  to be a formal and hence actual solution of the differential equation for
  $p_4$.  Together with the boundary conditions supplied by Theorem
  \ref{thm:p4at0} this proves the claim.
\end{proof}

\begin{remark}\label{rk:p4hypform}
  Let us indicate how the hypergeometric expression for $p_4$ given in
  Theorem~\ref{thm:p4hyp} was discovered. Consider the generating series
  \begin{equation}\label{eq:y0def}
    y_0(z) = \sum_{k=0}^\infty W_4(2k) z^k
  \end{equation}
  of the Domb numbers which is just a rescaled version of \eqref{eq:p4de-ana}.
  Corresponding to \eqref{eq:p4de-ana2}, the hypergeometric form for this
  series given in \cite{rogers-5f4} is
  \begin{equation}\label{eq:y0hyp}
    y_0(z) = \frac{1}{1-4z}\,\pFq32{\frac13,\frac12,\frac23}{1,1}{\frac{108z^2}{(1-4z)^3}}
  \end{equation}
  which converges for $|z|<1/16$.
  $y_0$ satisfies the differential equation $B_4 \cdot y_0(z) = 0$ where
  \begin{equation}\label{eq:B4}
    B_4 = 64z^2 (\theta+1)^3 - 2z (2\theta+1)(5\theta^2+5\theta+2) + \theta^3
  \end{equation}
  and $\theta = z \frac{\md}{\md z}$. Up to a change of variables this is
  \eqref{eq:p4de}; $y_0$ is the unique solution which is analytic at zero and
  takes the value $1$ at zero; the other solutions which are not a multiple of
  $y_0$ have a single or double logarithmic singularity. Let $y_1$ be the
  solution characterized by
  \begin{equation}\label{eq:y1def}
    y_1(z) - y_0(z) \log(z) \in z \QQ[[z]].
  \end{equation}
  Note that it follows from \eqref{eq:y1def} as well as Theorem~\ref{thm:p4at0}
  together with the initial values $s_{4,0} = \frac3{2\pi^2}$ and $r_{4,0} =
  s_{4,0} \log(8)$ that $p_4$, for small positive argument, is given by
  \begin{equation}\label{eq:p4y1}
    p_4(x) = - \frac{3x}{4\pi^2} \; y_1\left( \frac{x^2}{64} \right).
  \end{equation}
  If $x\in(2,4)$ and $z=x^2/64$ then the argument $t=\frac{108z^2}{(1-4z)^3}$
  of the hypergeometric function in \eqref{eq:y0hyp} takes the values
  $(1,\infty)$. We therefore consider the solutions of the corresponding
  hypergeometric equation at infinity. A standard basis for these is
  \begin{align}
    t^{-1/3}\pFq32{\frac13,\frac13,\frac13}{\frac23,\frac56}{\frac1t}, \quad
    t^{-1/2}\pFq32{\frac12,\frac12,\frac12}{\frac56,\frac76}{\frac1t}, \quad
    t^{-2/3}\pFq32{\frac23,\frac23,\frac23}{\frac43,\frac76}{\frac1t}.
  \end{align}
  In fact, the second element suffices to express $p_4$ on the interval $(2,4)$
  as shown in Theorem~\ref{thm:p4hyp}.
  \qede
\end{remark}

We close this section by showing that, remarkably, $p_4$ has  modular
  structure.

\begin{remark}\label{rk:p4modularity}
  As shown in \cite{cz-apery} the series $y_0$ defined in
  \eqref{eq:y0def} possesses the modular parameterization
  \begin{equation}\label{eq:y0modular}
    y_0\left( -\frac{\eta(2\tau)^6\eta(6\tau)^6}{\eta(\tau)^6\eta(3\tau)^6} \right)
    = \frac{\eta(\tau)^4\eta(3\tau)^4}{\eta(2\tau)^2\eta(6\tau)^2}.
  \end{equation}
  Here $\eta$ is the \emph{Dedekind eta function} defined as
  \begin{equation}\label{eq:eta}
    \eta(\tau) = q^{1/24}\, \prod_{n=1}^\infty(1-q^n)
    = q^{1/24}\,\sum_{n=-\infty}^\infty (-1)^n q^{n(3n+1)/2},
  \end{equation}
  where $q=e^{2\pi i\tau}$.
  Moreover, the quotient of the logarithmic solution $y_1$ defined in
  \eqref{eq:y1def} and $y_0$ is related to the modular parameter $\tau$ used in
  \eqref{eq:y0modular} by
  \begin{equation}\label{eq:y1y0}
    \exp\left( \frac{y_1(z)}{y_0(z)} \right) =e^{(2\tau+1)\pi i} = -q.
  \end{equation}

  Combining \eqref{eq:y0modular}, \eqref{eq:y1y0} and \eqref{eq:p4y1} one
  obtains the modular representation
  \begin{equation}\label{eq:p4modular}
    p_4\left( 8i\frac{\eta(2\tau)^3\eta(6\tau)^3}{\eta(\tau)^3\eta(3\tau)^3} \right)
    = \frac{6(2\tau+1)}{\pi} \eta(\tau)\eta(2\tau)\eta(3\tau)\eta(6\tau)
  \end{equation}
  valid when the argument of $p_4$ is small and positive. This is the case for
  $\tau = -1/2 + i y$ when $y>0$. Remarkably, the argument attains the value
  $1$ at the quadratic irrationality $\tau = (\sqrt{-5/3}-1)/2$ (the $5/3$rd
  singular value of the next section).  As a consequence, the value $p_4(1)$
  has a nice evaluation which is given in Theorem~\ref{thm:r50cs}.
  \qede
\end{remark}

\section{The density $p_5$}\label{sec:p5}

As shown in \cite{bsw-rw2}, $W_5 (s)$ has simple poles at $- 2, - 4,
\ldots$, compare Figure \ref{fig:W5}. We write $r_{5, k} = \Res_{- 2
k - 2} W_5$ for the residue of $W_5$ at $s=-2k-2$.  A surprising
bonus is an evaluation of $r_{5,0}=p_4(1) \approx 0.3299338011$, the
residue at $s=-2$. This is because in general for $n \ge 4$, one has
$$\Res_{- 2} W_{n+1}=p_{n+1}'(0)=p_n(1),$$ as follows from
\cite[Prop. 1(b)]{bsw-rw2}; here $p_{n+1}'(0)$ denotes the
derivative from the right at zero.

Explicitly, using Theorem \ref{thm:p4hyp04}, we have,
\begin{equation}\label{eq:res50}
  r_{5,0} =p_5'(0) = \frac{2\sqrt{15}}{\pi^2}\,
    {\rm Re}\,{\pFq32{\frac12, \frac12, \frac12}{\frac56,\frac76}{\frac{125}4}}.
\end{equation}

In fact, based on the modularity of $p_4$ discussed in
Remark~\ref{rk:p4modularity} we find:

\begin{theorem}\label{thm:r50cs}
  \begin{equation}\label{eq:r50cs}
    r_{5,0} =
    \frac{1}{2\pi^2}
    \sqrt{\frac{\Gamma(\frac{1}{15})\Gamma(\frac{2}{15})\Gamma(\frac{4}{15})\Gamma(\frac{8}{15})}
    {5\Gamma(\frac{7}{15})\Gamma(\frac{11}{15})\Gamma(\frac{13}{15})\Gamma(\frac{14}{15})}}.
  \end{equation}
\end{theorem}

\begin{proof}
  The value $\tau = (\sqrt{-5/3}-1)/2$ in \eqref{eq:p4modular} gives the value
  $p_4(1) = r_{5,0}$.  Applying the Chowla--Selberg formula
  \cite{sc67,borwein-piagm} to evaluate the eta functions yields the claimed
  evaluation.
\end{proof}

Using \cite[Table 4, (ii)]{bz92},
\eqref{eq:r50cs} may be simplified to
\begin{align}\label{eq:r50g2}
  r_{5,0} &= \frac{\sqrt {5}}{40}\,\frac{\Gamma(\frac{1}{15})\Gamma(\frac{2}{15})\Gamma(\frac{4}{15})\Gamma(\frac{8}{15})}{\pi^4 } \\
  &= \frac{3\sqrt {5} }{\pi^3}\,\frac{\left( \sqrt {5}-1 \right)}{2} \,K^2_{15} = \frac{\sqrt{15}}{\pi^3} \,K_{5/3} K_{15},
\end{align}
where $K_{15}$ and $K_{5/3}$ are the complete elliptic integral at the 15th and
$5/3$rd singular values \cite{borwein-piagm}.



Remarkably, these evaluations appear to extend to
$r_{5,1}\approx0.006616730259$, the residue at $s=-4$. Resemblance
to \emph{the tiny nome of Bologna} \cite{bbbg-bessel} led us to
discover
--- and then check to 400 places using \eqref{r50} and \eqref{r51} --- that
\begin{align}\label{eq:r51g}
  r_{5,1} &\stackrel{?}{=} {\frac {13}{1800\sqrt {5}}}\,{\frac {\Gamma(\frac{1}{15})\Gamma(\frac{2}{15})\Gamma(\frac{4}{15})\Gamma(\frac{8}{15})}{{\pi }^{4}}}
  -\frac1{\sqrt {5}}\,{\frac {\Gamma(\frac{7}{15})\Gamma(\frac{11}{15})\Gamma(\frac{13}{15})\Gamma(\frac{14}{15})}{{\pi }^{4}}}.
\end{align}
Using \eqref{eq:r50g2} this evaluation can be neatly restated as
\begin{align}
  r_{5,1} \stackrel{?}{=} \frac{13}{225} r_{5,0} -\frac2{5\pi^4}\frac{1}{r_{5,0}}.
\end{align}

We summarize our knowledge as follows:

\begin{theorem}\label{thm:p5}
  The density $p_5$ is real analytic on $(0,5)$ except at $1$ and $3$
  and satisfies the differential equation $A_5 \cdot p_5(x) = 0$
  where $A_5$ is the operator
  \begin{align}\label{eq:p5de}
    A_5 &= x^6 (\theta+1)^4
    - x^4 (35\theta^4 + 42\theta^2 + 3) \\
    &\quad+ x^2 (259(\theta-1)^4 + 104(\theta-1)^2)
    - \left( 15(\theta-3)(\theta-1) \right)^2 \nonumber
  \end{align}
  and $\theta=x D_x$.
  Moreover, for small $x>0$,
  \begin{equation}\label{eq:p5at0}
    p_5 (x) = \sum_{k = 0}^\infty r_{5, k}\, x^{2 k + 1}
  \end{equation}
  where
  \begin{align}\label{eq:res5rec}
    \left( 15 (2k+2) (2k+4) \right)^2 r_{5,k+2}
    &= \left( 259(2k+2)^4 + 104(2k+2)^2 \right) r_{5,k+1} \nonumber\\
    &\quad- \left( 35(2k+1)^4 + 42(2k+1)^2 + 3 \right) r_{5,k}
    + (2k)^4 r_{5,k-1}
  \end{align}
  with explicit initial values  $r_{5,-1}=0$ and $r_{5,0}$, $r_{5,1}$
  given by \eqref{eq:r50g2} and \eqref{eq:r51g} above.
\end{theorem}

\begin{proof}
First, the differential equation \eqref{eq:p5de} is computed as was
that for $p_4$, see \eqref{eq:p4de}. Next, as detailed in \cite[Ex.
3]{bsw-rw2} the residues satisfy the recurrence relation
\eqref{eq:res5rec} with the given initial values. Finally,
proceeding as for \eqref{eq:p4at0}, we deduce that \eqref{eq:p5at0}
holds for small $x>0$.
\end{proof}

 Numerically, the series \eqref{eq:p5at0} appears to converge for $|x|<3$ which
is in accordance with $\tfrac19$ being a root of the characteristic
polynomial of the recurrence \eqref{eq:res5rec}; see also
\eqref{eq:deleadingcoeff}. The series \eqref{eq:p5at0} is depicted
in Figure \ref{fig:p5x}.

\begin{figure}[htbp]
  \begin{center}
    \includegraphics[width=0.5\textwidth]{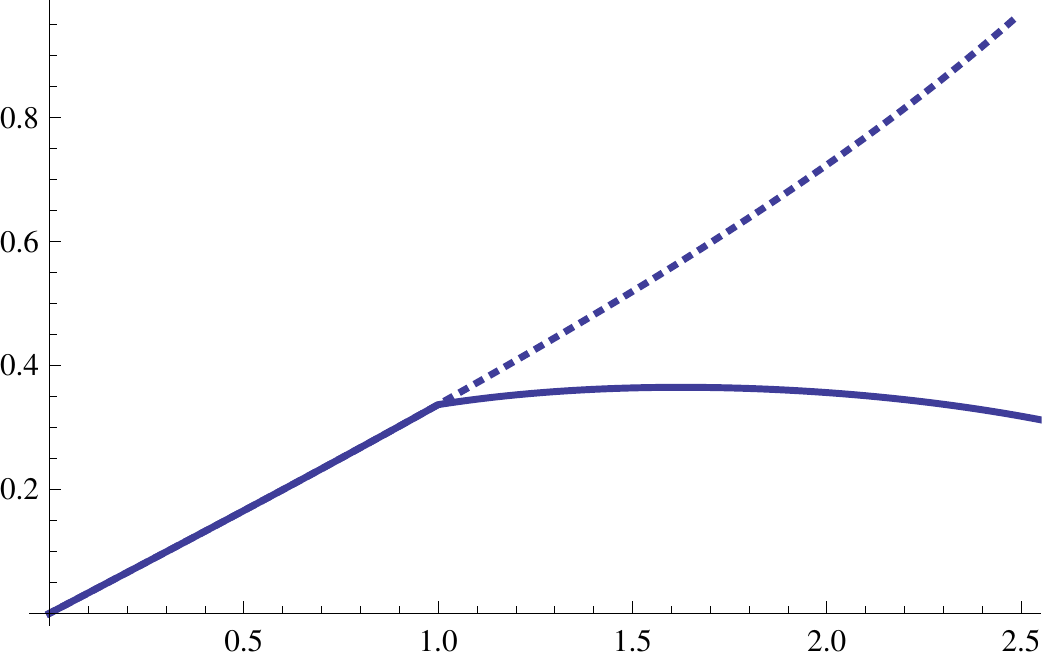}
    \caption{The series \eqref{eq:p5at0} (dotted) and $p_5$.}
    \label{fig:p5x}
  \end{center}
\end{figure}

Since the poles of $W_5$ are simple, no logarithmic terms are
involved in \eqref{eq:p5at0} as opposed to \eqref{eq:p4at0}. In particular, by
computing a few more residues from \eqref{eq:res5rec},
\[ p_5(x) = 0.329934\,x + 0.00661673\,{x}^{3} + 0.000262333\,{x}^{5} + 0.0000141185\,{x}^{7} + O(x^9) \]
near 0 (with each coefficient given to six digits of precision only),
explaining the strikingly straight shape of $p_5(x)$ on $[0,1]$. This
phenomenon was observed by Pearson \cite{Pea06} who stated that for $p_5(x)/x$
between $x=0$ and $x=1$,
\begin{quote}
  ``the graphical construction, however carefully reinvestigated, did not
  permit of our considering the curve to be anything but a \emph{straight}
  line\ldots Even if it is not absolutely true, it exemplifies the
  extraordinary power of such integrals of $J$ products [that is,
  \eqref{eq:pnbessel}] to give extremely close approximations to such simple
  forms as horizontal lines.''
\end{quote}
This conjecture was investigated in detail in \cite{fettis63} wherein the
nonlinearity was first rigorously established.  This work and various more
recent papers highlight the difficulty of computing the underlying Bessel
integrals.

\begin{remark}
  Recall from Example \ref{eg:W4p4mellinat0} that the asymptotic behaviour of
  $p_n$ at zero is determined by the poles of the moments $W_n(s)$.  To obtain
  information about the behaviour of $p_n(x)$ as $x \to n^-$, we consider the
  ``reversed'' densities $\tilde{p}_n(x) = p_n(n-x)$ and their moments
  $\tilde{W}_n(s)$. For non-negative integers $k$,
  \begin{equation*}
    \tilde{W}_n(k) = \int_0^n x^k \tilde{p}_n(x) \id x
    = \int_0^n (n-x)^k p_n(x) \id x
    = \sum_{j=0}^k \binom{k}{j} (-1)^j n^{k-j} W_n(j)
  \end{equation*}
  On the other hand, we can find a recurrence satisfied by the $\tilde{W}_n(s)$
  as follows: a differential equation for the densities $\tilde{p}_n(x)$ is
  obtained from Theorem \ref{thm:pnde} by a change of variables. The Mellin
  transform method as described in Example \ref{eg:p4de} then provides a
  recurrence for the moments $\tilde{W}_n(s)$.  We next apply the same
  reasoning as in \cite{bsw-rw2} to obtain information about the pole structure of
  $\tilde{W}_n(s)$. It should be emphasized that this involves knowledge about
  initial conditions in term of explicit values of initial moments $W_n(2k)$.

  For instance, in the case $n=4$, we find that the moments $\tilde{W}_4(s)$
  have simple poles at $-\tfrac32, -\tfrac52, -\tfrac72, \ldots$ which
  predicts an expansion of $p_4(x)$ as given in Theorem \ref{thm:p4at4}.

  For $n=5$, we learn that $\tilde{W}_5(s)$ has simple poles at
  $s=-2,-3,-4,\ldots$. It then follows, as for \eqref{eq:p5at0}, that $p_5 (x) =
  \sum_{k = 0}^\infty \tilde{r}_{5, k}\, (x-5)^{k + 1}$ for $x\le5$ and close
  to $5$. The $\tilde{r}_{5,k}$ are the residues of $\tilde{W}_5(s)$ at
  $s=-k-2$.
  \qede
\end{remark}

\section{Derivative evaluations of $W_n$}\label{sec:Wndiff}

As illustrated by Theorem \ref{thm:p4at0}, the residues of $W_n(s)$ are very
important for studying the densities $p_n$ as they directly translate into
behaviour of $p_n$ at $0$. The residues may be obtained as a linear combination of
the values of $W_n(s)$ and $W_n'(s)$.

\begin{example}[Residues of $W_n$]
  Using the functional equation for $W_3(s)$ and L'H\^{o}pital's rule we find that
  the residue at $s=-2$ can be expressed as
  \begin{equation}\label{eq:W3res}
    \Res_{-2}(W_3) = \frac{8+12W_3'(0)-4W_3'(2)}9.
  \end{equation}
  This is a general principle and we likewise obtain for instance:
  \begin{align}\label{r50}
    \Res_{-2}(W_5) &= \frac{16+1140W_5'(0)-804W_5'(2)+64W_5'(4)}{225},\\\label{r51}
    \Res_{-4}(W_5) &= \frac{26 \Res_{-2}(W_5) -16 -20 W_5'(0) + 4 W_5'(2)}{225}.
  \end{align}
  In the presence of double poles, as for $W_4$,
  \begin{equation}\label{w4coeff}
     \lim_{s \to -2}(s+2)^2 W_4(s) = \frac{3+4W_4'(0)-W_4'(2)}8
  \end{equation}
  and for the residue:
  \begin{equation}\label{w4resid}
    \Res_{-2}(W_4) = \frac{9+18W_4'(0)-3W_4'(2)+4W_4''(0)-W_4''(2)}{16}.
  \end{equation}
  Equations (\ref{w4coeff}, \ref{w4resid}) are used in Example \ref{eg:W4p4mellinat0} and
  each unknown is evaluated below.
  \qede
\end{example}

We are therefore interested in evaluations of the derivatives of $W_n$ for even
arguments.

\begin{example}[Derivatives of $W_3$ and $W_4$]
  Differentiating the double integral for $W_3(s)$ (\ref{def:Wns}) under
  the integral sign, we have
  \begin{equation*}
    W_3'(0) = \frac12 \int_0^1 \int_0^1 \log(4\sin(\pi y)\cos(2\pi x)+3-2\cos(2\pi y)) \id x \id y.
  \end{equation*}
  Then, using \[\int_0^1 \log(a+b\cos(2\pi x)) \id x =
  \log\left(\frac12\left(a+\sqrt{a^2-b^2}\right)\right) \ \mathrm{for} \ a >b >0,\] we deduce
  \begin{equation}\label{w3-cl}
    W_3'(0) = \int_{1/6}^{5/6} \log(2\sin(\pi y)) \id y = \frac{1}{\pi}\,\Cl\left(\frac{\pi}{3}\right),
  \end{equation}
  where $\Cl$ denotes the \emph{Clausen} function.  Knowing as we do that the
  residue at $s=-2$ is $2/(\sqrt{3}\pi)$, we can thus also obtain from (\ref{eq:W3res}) that
  \[ W_3'(2)=2+ \frac{3}{\pi} \Cl\left(\frac{\pi}{3}\right) - \frac{3\sqrt{3}}{2\pi}. \]

  In like fashion,
  \begin{align}\label{w4diffm1}
    W_4'(0) & = \frac 3{8\pi^2} \int_0^\pi \int_0^\pi
      \log\left(3+2\,\cos x+2\,\cos y +2\,\cos(x-y) \right) \id x \id y \nonumber\\
    & = \frac{7}{2}\frac{\zeta(3)}{\pi^2}.
  \end{align}
  The final equality will be shown in Example \ref{ex:w340}. Note that we may
  also write
  \begin{equation*}
    W_3'(0) = \frac{1}{8\pi^2} \int_0^{2\pi}\int_0^{2\pi}
      \log(3+2\cos x+2\cos y+2\cos(x-y)) \id x \id y.
  \end{equation*}
  The similarity between $W_3'(0)$ and $W_4'(0)$ is not coincidental, but comes
  from applying
  \begin{equation*}
    \int_0^1 \log\left((a+\cos 2\pi x)^2 + (b+\sin 2\pi x)^2\right) \id x
    = \left\{ \begin{array}{ll}
     \log(a^2+b^2) & \text{if $a^2+b^2>1$,} \\
     0 & \text{otherwise}
   \end{array} \right.
  \end{equation*}
  to the triple integral of $W_4'(0)$. As this reduction breaks the symmetry,
  we cannot apply it to $W_5'(0)$ to get a similar integral.
  \qede
\end{example}

In general, differentiating the Bessel integral expression
\begin{equation}\label{eq-broadhurst}
  W_n(s) = 2^{s+1-k} \frac{\Gamma(1+\tfrac{s}{2})}{\Gamma(k-\tfrac{s}{2})}
    \ift x^{2k-s-1} \left(-\frac{1}{x}\frac{\md}{\md x}\right)^k J_0^n(x) \id x,
\end{equation}
obtained by David Broadhurst \cite{broadhurst-rw} and discussed in \cite{bsw-rw2},
under the integral sign gives
\begin{align}\label{eq:wnd0}
  W_n'(0) & = n \ift \left(\log\left(\frac 2x  \right)-\gamma\right)
    J_0^{n-1}(x)J_1(x) \id x \nonumber\\
  & = \log(2)-\gamma-n \ift \log(x)J_0^{n-1}(x)J_1(x) \id x,
\end{align}
where $\gamma$ is the \emph{Euler-Mascheroni} constant, and
\begin{equation*}
  W_n''(0) = n \ift \left(\log\left(\frac 2x  \right)-\gamma\right)^2
    J_0^{n-1}(x)J_1(x) \id x.
\end{equation*}
Likewise
\begin{equation*}
  W_n'(-1) = (\log(2) -\gamma)W_n(-1)-\ift \log(x) J_0^n(x) \id x,
\end{equation*}
and
\begin{equation*}
  W_n'(1) = \ift \frac{n}{x} J_0^{n-1}(x)J_1(x)\left(1-\gamma-\log(2x)\right) \id x.
\end{equation*}

\begin{remark}
  We may therefore obtain many identities by comparing the above equations to
  known values. For instance,
  \begin{equation*}
    3 \ift \log(x)J_0^2(x)J_1(x) \id x
    = \log(2)-\gamma-\frac{1}{\pi}\Cl\left( \frac{\pi}{3}\right).
  \end{equation*}
  \vskip-\baselineskip\qede
\end{remark}

\begin{example}[Derivatives of $W_5$]
  In the case $n=5$,
  \begin{equation*}
    W_5'(0) = 5 \ift \left( \log  \left( \frac 2t \right) -\gamma \right)
      J_0^4(t) J_1(t) \id t \approx 0.54441256
  \end{equation*}
  with similar but more elaborate formulae for $W_5'(2)$ and $W_5'(4)$.
  Observe that  in general we also have
  \begin{equation}
    W_n'(0) = \log(2)-\gamma-\int_0^1 \left(J_0^n(x)-1\right)  \frac{\md x}{x}
    - \int_1^\infty J_0^n(x) \frac{\md x}{x},
  \end{equation}
 which is useful numerically. \qede
\end{example}

In fact, the hypergeometric representation of $W_3$ and $W_4$ first obtained
in \cite{crandall-rw} and recalled below also makes derivation of the derivatives
of $W_3$ and $W_4$ possible.

\begin{corollary}[Hypergeometric forms]\label{cor:W34hyp}
  For $s$ not an odd integer, we have
  \begin{align}
    W_3(s) &= \frac{1}{2^{2s+1}} \tan\left(\frac{\pi s}{2}\right)
      \binom{s}{\frac{s-1}{2}}^2
      \pFq32{\frac12, \frac12, \frac12}{\frac{s+3}{2}, \frac{s+3}{2}}{\frac14}
    + \binom{s}{\frac{s}{2}} \pFq32{-\frac s2, -\frac s2, -\frac s2}{1, -\frac{s-1}{2}}{\frac14}, \label{eq:cf3s}\\
   \shortintertext{and, if also $\Re s > -2$, we have}
    W_4(s) &= \frac{1}{2^{2s}} \tan\left(\frac{\pi s}{2}\right)
      \binom{s}{\frac{s-1}{2}}^3
      \pFq43{\frac12, \frac12, \frac12 , \frac s2 +1}{\frac{s+3}{2},\frac{s+3}{2},\frac{s+3}{2}}{1}
    + \binom{s}{\frac{s}{2}} \pFq43{\frac12,-\frac s2, -\frac s2, -\frac s2}{1,1,-\frac{s-1}{2}}{1}. \label{eq:cf4s}
  \end{align}
\end{corollary}

\begin{example}[Evaluation of $W_3'(0)$ and $W_4'(0)$] \label{ex:w340}
If we write  (\ref{eq:cf3s}) or (\ref{eq:cf4s}) as
$W_n(s)=f_1(s)F_1(s)+f_2(s)F_2(s)$, where $F_1, F_2$ are the corresponding
hypergeometric functions, then it can be readily verified that $f_1(0) =
f_2'(0) = F_2'(0) =0$. Thus, differentiating (\ref{eq:cf3s}) by appealing to
the product rule we get:
\begin{equation*}
  W_3'(0) = \frac{1}{\pi} \pFq32{\frac12, \frac12, \frac12}{\frac32, \frac32}{\frac14}
  = \frac{1}{\pi} \Cl \left( \frac{\pi}{3} \right).
\end{equation*}
The last equality follows from setting $\theta=\pi/6$ in the identity
\begin{equation*}
  2\sin(\theta) \pFq32{\frac12, \frac12, \frac12}{\frac32, \frac32}{\sin^2 \theta}
  = \Cl\left( 2\,\theta \right) +2\,\theta \log \left( 2\sin \theta \right).
\end{equation*}

Likewise, differentiating  (\ref{eq:cf4s}) gives
\begin{equation}
  W_4'(0) = \frac{4}{\pi^2} \pFq43{\frac12, \frac12, \frac12, 1}{\frac32, \frac32, \frac32}{1}
  = \frac{7 \zeta(3)}{2\pi^2},
\end{equation}
thus verifying (\ref{w4diffm1}). In this case the hypergeometric
evaluation
\begin{equation*}
  \pFq43{\frac12, \frac12, \frac12, 1}{\frac32, \frac32, \frac32}{1}
  = \sum_{n=0}^\infty \frac{1}{(2n+1)^3}= \frac 78 \zeta(3),
\end{equation*}
is elementary.
\qede
\end{example}

Differentiating (\ref{eq:cf3s}) at $s=2$ leads to the evaluation
\begin{equation*}
  \pFq32{\frac12, \frac12, \frac12}{\frac52, \frac52}{\frac14}
  = \frac{27}{4}\left(\Cl\left(\frac\pi 3\right)-\frac{\sqrt{3}} 2\right),
\end{equation*}
while from (\ref{eq:cf4s}) at $s=2$ we obtain
\begin{equation}
  W_4'(2) = 3 + \frac{14 \zeta(3) - 12}{\pi^2}.
\end{equation}
Thus we have enough information to evaluate (\ref{w4coeff}) (with the answer $3/(2\pi^2)$).

Note that with two such starting values, all derivatives of $W_3(s)$ or $W_4(s)$
at even $s$ may be computed recursively.

We also note here that the same technique yields
\begin{align}\label{eq:w3dd}
  W_3''(0) & = \frac{\pi^2}{12} - \frac{2}{\pi} \sum_{n=0}^\infty
    \frac{\binom{2n}{n}}{4^{2n}}\frac{ H_{n+1/2}}{(2n+1)^2} \\
  & = \frac{\pi^2}{12} + \frac{4 \log (2)}{\pi}\Cl\left(\frac\pi 3\right)
  - \frac{4}{\pi} \sum_{n=0}^\infty \frac{\binom{2n}{n}}{4^{2n}}
    \frac{\sum_{k=0}^n \frac 1{2k+1}}{(2n+1)^2},
\end{align}
and, quite remarkably,
\begin{align}\label{eq:w4dd}
  W_4''(0) & = \frac{\pi^2}{12}+\frac{7\zeta(3)\log (2)}{\pi^2}
  + \frac{4}{\pi^2} \sum_{n=0}^\infty \frac{H_n-3H_{n+1/2}}{(2n+1)^3} \\
  & = \frac{24{\rm Li}_4\left(\frac 12\right)-18\zeta(4)+21\zeta(3) \log (2)
  - 6\zeta(2)\log^2 (2) +\log^4 (2)}{\pi^2},\nonumber
\end{align}
where the very final evaluation is obtained from results in \cite[\S5]{bzb}.
Here ${\rm Li}_4(1/2)$ is the \emph{polylogarithm} of order 4, while $H_n
:=\gamma+\Psi(n+1)$ denotes the $n$th harmonic number, where $\Psi$ is the
\emph{digamma} function. So for non-negative integers $n$, we
have explicitly $H_n= \sum_{k=1}^n 1/k$, as before, and
\begin{equation*}
  H_{n+1/2}= 2\sum_{k=1}^{n+1} \frac1{2k-1}-2\log (2).
\end{equation*}

An evaluation of $W_3''(0)$ in terms of polylogarithmic constants is given in
\cite{logsin1}. In particular, this gives an evaluation of the sum on the
right-hand side of \eqref{eq:w3dd}.

Finally, the corresponding sum for $W_4''(2)$ may be split into a telescoping
part and a part involving $W_4''(0)$. Therefore, it can also be written as a
linear combination of the constants used in (\ref{eq:w4dd}). In summary, we have
all the pieces to evaluate (\ref{w4resid}), obtaining the answer $9\log
(2)/(2\pi^2)$.

\subsection{Relations with Mahler measure}

For a (Laurent) polynomial $f (x_1, \ldots, x_n)$, its \emph{logarithmic Mahler
measure}, see for instance \cite{rtv-mahler}, is defined as
\[ m (f) = \int_0^1 \ldots \int_0^1 \log \left| f \left( e^{2 \pi i t_1},
   \ldots, e^{2 \pi i t_n} \right) \right| \md t_1 \cdots \md t_n. \]
Recall that the $s$th moments of an $n$-step random walk are given by
\[ W_n (s) = \int_0^1 \ldots \int_0^1 \left| \sum_{k = 1}^n e^{2 \pi
   i t_k} \right|^s \md t_1 \cdots \md t_n =\|x_1 +
   \ldots + x_n \|_s^s \]
where $\| \cdot \|_p$ denotes the $p$-norm over the unit $n$-torus, and hence
\[ W_n' (0) = m (x_1 + \ldots + x_n) = m (1 + x_1 + \ldots + x_{n-1}) . \]
Thus the derivative evaluations in the previous sections are also Mahler
measure evaluations. In particular, we rediscovered
\[ W_3' (0) = \frac{1}{\pi} \tmop{Cl} \left( \frac{\pi}{3} \right) = L'
   (\chi_{- 3}, - 1) = m (1 + x_1 + x_2), \]
along with
\[ W_4' (0) = \frac{7 \zeta (3)}{2 \pi^2} = m (1 + x_1 + x_2 + x_3) \]
which are both due to C. Smyth \cite[(1.1) and (1.2)]{rtv-mahler} with proofs  first published in \cite[Appendix 1]{boyd}.

With this connection realized, we find the following conjectural expressions
put forth by Rodriguez-Villegas, mentioned in different form in \cite{finch},
\begin{equation}\label{eq:vil1} W_5'(0) \stackrel{?}{=} \left(\frac{15}{4\pi^2}\right)^{5/2}\,
  \int_0^\infty \left\{\eta^3(e^{-3t})\eta^3(e^{-5t})
  +\eta^3(e^{-t})\eta^3(e^{-15t})\right\} t^3\,\md t \end{equation}
and
\begin{equation}\label{eq:vil2} W_6'(0) \stackrel{?}{=} \left(\frac{3}{\pi^2}\right)^{3}\,
  \int_0^\infty \,\eta^2(e^{-t})\eta^2(e^{-2t}) \eta^2(e^{-3t})\eta^2(e^{-6t})\,t^4\,\md t, \end{equation}
where $\eta$ was defined in~\eqref{eq:eta}.  We have confirmed numerically that
the evaluation of $W_5'(0)$ in (\ref{eq:vil1}) holds to $600$ places.
Likewise, we have confirmed  that (\ref{eq:vil2}) holds to $80$ places. Details
of these somewhat arduous confirmations are given in \cite{bb-combat}.

Differentiating the series expansion for $W_n(s)$ obtained in \cite{bnsw-rw}
term by term, we obtain
\begin{equation}\label{eq:WD0v2}
  W_n'(0) = \log (n) - \sum_{m=1}^\infty \frac{1}{2m}
    \sum_{k=0}^m \binom{m}{k} \frac{(-1)^k W_n(2k)}{n^{2k}}.
\end{equation}
On the other hand, from \cite{rtv-mahler} we find the strikingly similar
\begin{equation}\label{eq:WD0v}
  W_n'(0) = \frac12 \log (n) - \frac{\gamma}{2} - \sum_{m=2}^\infty \frac{1}{2m}
    \sum_{k=0}^m \binom{m}{k} \frac{(-1)^k W_n(2k)}{k! n^{k}}.
\end{equation}


Finally, we note that $W_n(s)$ itself is a special case of \textit{zeta Mahler
measure} as introduced recently in \cite{aka}.

\section{New results on the moments $W_n$}\label{sec:Wn}

From \cite{bbbg-bessel} equation (23), we have for $k>0$ even,
\begin{equation}\label{k3}
W_3(k) =  \frac{3^{k+3/2}}{\pi \, 2^k\, \Gamma(k/2+1)^2}  \int_0^\infty  t^{k+1}K_0(t)^2 I_0(t) \md t,
\end{equation}
where $I_0(t), K_0(t)$ denote the \textit{modified Bessel functions} of the first and second kind, respectively.

Similarly, \cite{bbbg-bessel} equation (55) states that for $k>0$ even,
\begin{equation}\label{k4}
W_4(k) = \frac{4^{k+2}}{\pi^2 \, \Gamma(k/2+1)^2} \int_0^\infty t^{k+1}K_0(t)^3 I_0(t) \md t.
\end{equation}

Equation \eqref{k3} can be formally reduced to a closed form as a $_3F_2$ (for
instance by \textit{Mathematica}). At $k=\pm 1$, the closed form agrees
with $W_3(\pm 1)$. As both sides of \eqref{k3} satisfy the same recursion
(\cite{bbbg-bessel} equation (8)), we see that it in fact holds for all
integers $k>-2$.

In the following we shall use Carlson's theorem (\cite{titchmarsh}) which states:

\emph{Let
$f$ be analytic in the right half-plane $\Re z \ge 0$ and of exponential type
with the additional requirement that $$|f(z)| \le M e^{d |z|}$$ for some $d <
\pi$ on the imaginary axis $\Re z = 0$. If $f(k)=0$ for $k=0,1,2,\ldots$ then
$f(z)=0$ identically}.
We then have the following:

\begin{lemma} Equation (\ref{k3}) holds for all $k$ with $\mathrm{Re}\, k>-2$.
\end{lemma}
\begin{proof}

Both sides of (\ref{k3}) are of exponential type and agree when $k=0,1,2, \ldots$. The standard estimate shows that the right-hand side grows like $e^{|y| \pi/2}$ on the imaginary axis. Therefore the conditions of Carlson's theorem are satisfied and the identity holds whenever the right-hand side converges.
\end{proof}

Using the closed form given by the computer algebra system, we thus have:

\begin{theorem}[Single hypergeometric for $W_3(s)$]
  For $s$ not a negative integer $<-1$,
  \begin{equation}\label{singleh}
    W_3(s) = \frac{3^{s+3/2}}{2\pi}\frac{\Gamma(1+s/2)^2}{\Gamma(s+2)}\,
      \pFq32{\frac{s+2}2, \frac{s+2}2,\frac{s+2}2}{1,\frac{s+3}2}{\frac14}.
  \end{equation}
\end{theorem}

Turning our attention to negative integers, we have for $k \ge 0$ an integer:
\begin{equation} \label{kn}
W_3(-2k-1) = \frac{4}{\pi^3} \left(\frac{2^k k!}{(2k)!}\right)^2 \int_0^\infty t^{2k}K_0(t)^3 \md t,
\end{equation}
because the two sides satisfy the same recursion (\cite[(8)]{bbbg-bessel}), and agree
when $k=0, 1$ (\cite[(47) and (48)]{bbbg-bessel}).

\begin{remark}
Equation (\ref{kn}) however does not hold when $k$ is not an integer. Also, combining (\ref{kn}) and (\ref{k3}) for $W_3(-1)$, we deduce
\[ \int_0^\infty K_0(t)^2 I_0(t)\, \md t = \frac{2}{\sqrt3 \pi} \int_0^\infty K_0(t)^3 \,\md t = \frac{\pi^2}{2\sqrt{3}} \int_0^\infty J_0(t)^3\, \md t. \]
\end{remark}

From (\ref{kn}), we experimentally determined a single hypergeometric for $W_3(s)$ at negative odd integers:
\begin{lemma} For $k \ge 0$ an integer,
\[ W_3(-2k-1) = \frac{\sqrt{3} \, \binom{2k}{k}^2}{2^{4k+1}3^{2k}} \, _3F_2\left( {{\frac12, \frac12, \frac12}\atop{k+1, k+1}}\bigg| \frac14 \right). \]
\end{lemma}
\begin{proof} It is easy to check that both sides agree at $k=0, 1$. Therefore we need only to show that they satisfy the same recursion. The recursion for the left-hand side implies a contiguous relation for the right-hand side, which can be verified by extracting the summand and applying Gosper's algorithm (\cite{aeqb}).
\end{proof}

The integral in (\ref{kn}) shows that $W_3(-2k-1)$ decays to 0 rapidly -- very roughly like $9^{-k}$ as $k \to \infty$ -- and so is never $0$ when $k$ is an integer.

To show that (\ref{k4}) holds for more general $k$ required more work. Using Nicholson's integral representation in \cite{watson-bessel}, \[I_0(t)K_0(t) = \frac{2}{\pi} \int_0^{\pi/2} K_0(2t \sin a) \,\md a,\]
the integral in (\ref{k4}) simplifies to
\begin{equation}\label{kk}
 \frac{2}{\pi} \int_0^{\pi/2} \int_0^\infty t^{k+1} K_0(t)^2 K_0(2t \sin a) \,\md t \md a.
\end{equation}
The inner integral in (\ref{kk}) simplifies in terms of a \emph{Meijer G-function}; \textit{Mathematica} is able to produce
\[ \frac{\sqrt{\pi}}{8 \sin^{k+2} a}\,G_{3,3}^{3,2} \left( {{-\frac12, -\frac12, \frac12}\atop{0,0,0}}\bigg| \frac{1}{\sin^2 a}\right), \]
which transforms to
\[ \frac{\sqrt{\pi}}{8\sin^{k+2} a}\,G_{3,3}^{2,3} \left( {{1,1,1}\atop{\frac32,\frac32,\frac12}}\bigg| \sin^2 a\right). \]
Let $t = \sin^2 a$ in the above, so the outer integral in (\ref{kk}) transforms to
\begin{equation}\label{preeuler}
\frac{\sqrt{\pi}}{16} \int_0^1 t^{-\frac{k+3}2}(1-t)^{-\frac12}\, G_{3,3}^{2,3} \left( {{1,1,1}\atop{\frac32,\frac32,\frac12}}\bigg| t \right)
\,\md t.
\end{equation}
We can resolve this integral by applying the Euler-type integral
\begin{equation}\label{euler}
\int_0^1 t^{-a}(1-t)^{a-b-1} \, G_{p,q}^{m,n} \left({\mathbf{c}\atop\mathbf{d}} \bigg| z t \right) \md t = \Gamma(a-b)\, G_{p+1,q+1}^{m,n+1} \left({{a,\mathbf{c}}\atop{\mathbf{d},b}} \bigg| z \right).
\end{equation}
Indeed, when $k=-1$, the application of (\ref{euler}) recovers the Meijer G representation of $W_4(-1)$ (\cite{bsw-rw2}); that is, (\ref{k4}) holds for $k=-1$.

When $k=1$, the resulting Meijer G-function is
\[  G_{4,4}^{2,4} \left( {{2,1,1,1}\atop{\frac32,\frac32,\frac12,\frac32}}\bigg| 1 \right), \]
to which we apply Nesterenko's theorem (\cite{nest}), deducing a triple integral (up to a constant factor) for it:
\[ \int_0^1\int_0^1\int_0^1 \sqrt{\frac{x(1-x)z}{y(1-y)(1-z)(1-x(1-y z))^3}}\,\md x \md y\md z. \]
We can reduce the triple integral to a single integral,
\[ \int_0^1 \frac{8 E'(t)\big((1+t^2)K'(t)-2E'(t)\big)}{(1-t^2)^2} \,\md t. \]
Now applying the change of variable $t \mapsto (1-t)/(1+t)$, followed by quadratic transformations for $K$ and $E$, we finally get
\[ \int_0^1 \frac{4(1+t)}{t^2}E\left(\frac{2\sqrt{t}}{1+t}\right)\big(K(t)-E(t)\big) \,\md t, \]
which is, indeed, (a correct constant multiple times) the expression for $W_4(1)$ which follows from Section 3.1 in \cite{bsw-rw2}.

We finally observe that both sides of (\ref{k4}) satisfy the same recursion (\cite{bbbg-bessel} equation (9)), hence they agree for $k=0, 1, 2, \ldots$. Carlson's theorem applies with the same growth on the imaginary axis as in (\ref{k3}) and we have proven the following:
\begin{lemma}\label{L4} Equation (\ref{k4}) holds for all $k$ with $\mathrm{Re}\, k>-2$.
\end{lemma}

\begin{theorem}[Alternative Meijer G representation for $W_4(s)$] For all $s$,
\begin{equation}
W_4(s) = \frac{2^{2s+1}}{\pi^2 \, \Gamma(\frac12(s+2))^2} \, G_{4,4}^{2,4}\left( {{1,1,1,\frac{s+3}{2}}\atop{\frac{s+2}2,\frac{s+2}2,\frac{s+2}2,\frac12}}\bigg| 1\right).
\end{equation}
\end{theorem}
\begin{proof} Apply (\ref{euler}) to (\ref{preeuler}) for general $k$, and equality holds by Lemma \ref{L4}.
\end{proof}

Note that Lemma \ref{L4} combined with the known formula for $W_4(-1)$ in \cite{bsw-rw2} gives
\[ \frac{4}{\pi^3} \int_0^\infty K_0(t)^3I_0(t)\, \md t = \int_0^\infty J_0(t)^4\, \md t. \]

Armed with the knowledge of Lemma \ref{L4}, we may now resolve a very special but central case (corresponding to $n=2$) of Conjecture 1 in \cite{bsw-rw2}.
\begin{theorem}
For integer $s$,
\begin{equation} \label{conj}
W_4(s) = \sum_{j=0}^\infty \binom{s/2}{j}^2 W_3(s-2j).
\end{equation}
\end{theorem}
\begin{proof}
In \cite{bnsw-rw} it is shown that both sides satisfy the same three term recurrence, and agree when $s$ is even. Therefore, we only need to show that the identity holds for two consecutive odd values of $s$.

For $s=-1$, the right-hand side of (\ref{conj}) is
\[ \sum_{j=0}^\infty \binom{-1/2}{j}^2 W_3(-1-2j) = \sum_{j=0}^\infty \frac{2^{2-2j}}{\pi^3 j!^2} \int_0^\infty t^{2j} K_0(t)^3 \,\md t \]
upon using (\ref{kn}), and after interchanging summation and integration (which is justified as all terms are positive), this reduces to
\[ \frac{4}{\pi^3}  \int_0^\infty K_0(t)^3I_0(t)\, \md t, \]
which is the value for $W_4(-1)$ by Lemma \ref{L4}.

We note that the recursion for $W_4(s)$ gives the pleasing reflection property
\[ W_4(-2k-1) \, 2^{6k} = W_4(2k-1). \]
In particular, $W_4(-3) = \frac{1}{64} \, W_4(1)$. Now computing the right-hand side of (\ref{conj}) at $s=-3$, and interchanging summation and integration as before, we obtain
\[ \sum_{j=0}^\infty \binom{-3/2}{j}^2 W_3(-3-2j) = \frac{4}{\pi^3} \int_0^\infty t^2  K_0(t)^3 I_0(t) \,\md t = \frac{1}{64}W_4(1) =  W_4(-3). \]
Therefore (\ref{conj}) holds when $s=-1, -3$, and thus holds for all integer $s$.
\end{proof}

\bigskip
\paragraph{\textbf{Acknowledgments}}
We are grateful to David Bailey for significant numerical assistance, Michael
Mossinghoff for pointing us to the Mahler measure conjectures via
\cite{rtv-mahler}, and Plamen Djakov and Boris Mityagin for correspondence
related to Theorem \ref{thm:deleadingcoeffx} and the history of their proof.
We are specially grateful to Don Zagier for not only providing us with
his proof of Theorem \ref{thm:deleadingcoeffx} but also for his enormous
amount of helpful comments and improvements.
We also thank the reviewer for helpful suggestions.

The first and the last authors thankfully acknowledge support by the Australian
Research Council.


\section*{Appendix. \textbf{A family of combinatorial identities}}

\medskip
\hbox to\hsize{\hss \uppercase{Don Zagier}\footnote{The original note is unchanged.}\hss}


\vskip5mm

The ``collateral result'' of  Djakov and Mityagin, \cite{dm1, dm2}, is the pair
of identities
\begin{align*}
  &\sum_{\substack{ -m<i_1<\dots<i_k<m\\ i_2-i_1,\dots,i_k-i_{k-1}\ge2 }}
  \;\,\prod_{s=1}^k(m^2-i_s^2)
  = \sigma_k(1^2,3^2,\dots,(2m-1)^2)\,,
  \displaybreak[2]\\
  &\sum_{\substack{ 1-m<i_1<\dots<i_k<m\\ i_2-i_1,\dots,i_k-i_{k-1}\ge2 }}
  \;\,\prod_{s=1}^k(m-i_s)(m+i_s-1)
  = \sigma_k(2^2,4^2,\dots,(2m-2)^2)\,,
\end{align*}
where $m$ and $k$ are integers with $m\ge k\ge0$ and $\sigma_k$ denotes the $k$th elementary symmetric
function.  By setting $j_s=i_s+m$ in the first sum and
$j_s=i_s+m-1$ in the second, we can rewrite these formulas more uniformly as\footnote{Note that \eqref{eq:FMk} is precisely Theorem \ref{thm:deleadingcoeffx}.}
\begin{align}\label{eq:FMk}
F_{M,k}(M) = \begin{cases}
  \sigma_k(1^2,3^2,\dots,(M-1)^2)&\text{if $M$ is even,} \\
  \sigma_k(2^2,4^2,\dots,(M-1)^2)&\text{if $M$ is odd,}
\end{cases}
\end{align}
where $F_{M,k}(X)$ is the polynomial in $X$ (non-zero only if $M\ge2k\ge0$) defined by
\begin{align}\label{eq:FMkX}
  F_{M,k}(X) = \sum_{\substack{ 0<j_1<\dots<j_k<M\\ j_2-j_1,\dots,j_k-j_{k-1}\ge2 }}
  \prod_{s=1}^k j_s(X-j_s)\;.
\end{align}
%
The advantage of introducing the free variable $X$ in \eqref{eq:FMkX} is that
the functions $F_{M,k}(X)$ satisfy the recursion
\begin{align}\label{eq:FMkrec}
  F_{M+1,k+1}(X) - F_{M,k+1}(X) = M\,(X-M)\,F_{M-1,k}(X),
\end{align}
because the only paths that are counted on the left are those with $0<j_1<\dots<j_k<j_{k+1}=M$.

It is also advantageous to introduce the polynomial generating function
$$ \Phi_M = \Phi_M(X,u) = \sum_{0\le k\le M/2}(-1)^k\,F_{M,k}(X)\,u^{M-2k}\,,  $$
the first examples being
\begin{align*}
 &\Phi_0=1\,,\qquad \Phi_1=u\,,\qquad \Phi_2=u^2-(X-1)\,,\qquad \Phi_3=u^3-(3X-5)u\,,\\
 &\Phi_4=u^4 -(6X-14)u^2 +(3X^2-12X+9)\,,\\
 &\Phi_5=u^5-(10X-30)u^3+(15X^2-80X+89)\,,\\
 &\Phi_6= u^6 - (15X - 55)u^4 + (45X^2 - 300X + 439)u^2 - (15X^3 - 135X^2 + 345X - 225)\,.
\end{align*}
 In terms of this generating function, the recursion \eqref{eq:FMkrec} becomes
 \begin{align}\label{eq:Phirec}
 \Phi_{M+1} = u\,\Phi_M - M(X-M)\,\Phi_{M-1}
 \end{align}
and the identity \eqref{eq:FMk} to be proved can be written succinctly as
\begin{align}\label{eq:idPhi}
  \Phi_M(M,u) = \prod_{\substack{ |\lambda|<M\\ \lambda\not\equiv M\pmod2 }} (u-\lambda)\;.
\end{align}
Denote by $P_M(u)$ the polynomial on the right-hand side of \eqref{eq:idPhi}.  Looking for other pairs $(M,X)$
where $\Phi_M(X,u)$ has many integer roots, we find experimentally that this happens whenever $M-X$
is a non-negative integer, and studying the data more closely we are to conjecture the two formulas
\begin{align}\label{eq:Phi1}
 \Phi_M(M-n,u) = \frac1{2^n}\,\sum_{j=0}^n\binom nj\,P_M(u-n+2j)\qquad(M,\,n\ge0)
\end{align}
(a generalization of \eqref{eq:idPhi}) and
\begin{align}\label{eq:Phi2}
 \Phi_{M+n}(M,u) = \Phi_M(M,u)\,\Phi_n(-M,u)\qquad(M,\,n\ge0)\,.
\end{align}
Formula \eqref{eq:Phi2} is easy to prove, since it holds for $n=0$ trivially and for $n=1$ by~\eqref{eq:Phirec}
and since both sides satisfy the recursion $y_{n+1}=u\,y_n+n(M+n)\,y_{n-1}$ for $n=1,2,\dots$ by \eqref{eq:Phirec}.
On the other hand, combining \eqref{eq:idPhi}, \eqref{eq:Phi1} and \eqref{eq:Phi2} leads to the conjectural formula
$$  \Phi_n(-M,u) = \frac1{2^n}\,\sum_{j=0}^n\binom nj\,\frac{P_{M+n}(u-n+2j)}{P_n(u)}
 = n!\,\sum_{j=0}^n(-1)^j\binom{\frac{-u-M-1}2}j\,\binom{\frac{u-M-1}2}{n-j} $$
or, renaming the variables,
\begin{align}\label{eq:Phi3}
  \frac1{M!}\,\Phi_M(x+y+1,y-x) = \sum_{j=0}^M(-1)^j\binom xj\,\binom y{M-j}\;.
\end{align}
To prove this, we see by \eqref{eq:Phirec} that, denoting by $G_M=G_M(x,y)$ the expression on the right,
it suffices to prove the recursion $(M+1)G_{M+1}=(y-x)G_M+(M-x-y-1)G_{M-1}$. This is an easy binomial
coefficient identity, but once again it is easier to work with generating functions: the sum
\begin{align}\label{eq:sumG}
 \mathcal G(x,y;T) := \sum_{M=0}^\infty G_M(x,y) \,T^m = (1-T)^x\,(1+T)^y
\end{align}
satisfies the differential equation
$$\dfrac1{\mathcal G}\dfrac{\partial\mathcal G}{\partial T}=\dfrac y{1+T}-\dfrac x{1-T}$$
or
$$\dfrac{\partial\mathcal G}{\partial T}=(y-x)\,\mathcal G + \left(T\dfrac\partial{\partial T}-x-y\right)\mathcal G,$$
and this is equivalent to the desired recursion.

We can now complete the proof of \eqref{eq:FMk}.  Rewriting \eqref{eq:sumG} in the form
$$  \frac1{M!}\,\Phi_M(X,u) = \text{coeff}_{T^M}\left((1-T)^{\frac{X-u-1}2}\,(1+T)^{\frac{X+u-1}2}\right)\,,   $$
we find that, for $1\le j\le M$,
$$\frac1{M!}\, \Phi_M(M,M+1-2j) = \text{coeff}_{T^M}\left((1-T)^{j-1}\,(1+T)^{M-j}\right) = 0  $$
and hence that the polynomial on the left-hand side of \eqref{eq:idPhi} is divisible by the polynomial on the right,
which completes the proof since both are monic of degree~$M$ in~$u$.

\bibliography{densities-refs}
\bibliographystyle{alpha}

\end{document}